\documentclass[1 [leqno,11pt]{amsart}
\usepackage{amssymb, amsmath,latexsym,amsfonts,amsbsy,bbm, amsthm,mathtools,graphicx,CJKutf8,CJKnumb,CJKulem,color,xcolor,mathrsfs}

\usepackage{float}
\usepackage{hyperref}
\usepackage{braket}
\numberwithin{equation}{section}

\allowdisplaybreaks

\let\la=\lambda

\let\f=\frac

\def\R{\mathbb R}

\def\eps{\varepsilon}

\def\de{\delta}
\def\tcr{\textcolor{red}}

\newcommand{\beq}{\begin{equation}}
\newcommand{\eeq}{\end{equation}}
\newcommand{\ben}{\begin{eqnarray}}
\newcommand{\een}{\end{eqnarray}}
\newcommand{\beno}{\begin{eqnarray*}}
\newcommand{\eeno}{\end{eqnarray*}}

\newtheorem{theorem}{Theorem}[section]

\newtheorem{lemma}[theorem]{Lemma}
\newtheorem{proposition}[theorem]{Proposition}

\begin{document}

\title[Enhanced dissipation and Taylor dispersion]
{Enhanced dissipation and Taylor dispersion by a parallel shear flow in an infinite cylinder with unbounded cross section}

\author[T. Li]{Te Li}%
\address[T. Li]
 {Academy of Mathematics $\&$ Systems Science, The Chinese Academy of
Sciences, Beijing 100190, CHINA.}
\email{teli@amss.ac.cn}

\author[L. Zhang]{Le Zhang}%
\address[L. Zhang]
 {Academy of Mathematics $\&$ Systems Science, The Chinese Academy of
Sciences, Beijing 100190, CHINA.  } \email{zhangle211@mails.ucas.ac.cn}

\date{\today}

\begin{abstract}
In this paper, we investigate the long-time behavior of a passive scalar advected by a parallel shear flow in an infinite cylinder with unbounded cross section, in the regime where the viscosity coefficient satisfies $\nu \ll 1$, and in arbitrary spatial dimension. 
Under the assumption of an infinite cylinder, that is, $x \in \mathbb{R}$, the corresponding Fourier frequency $k$ (often referred to as the streamwise wave number) also ranges over the whole real line $\mathbb{R}$. 
In this setting, the enhanced dissipation phenomenon only occurs for high frequencies $\nu \le |k|$, whereas for low frequencies $|k| \le \nu$ only the decay of Taylor dispersion appears. 
It is worth noting that, in the case where $x \in \mathbb{T}$, the Fourier frequency does not contain low frequencies near zero, and thus enhanced dissipation occurs for all nonzero modes. Previously, Coti Zelati and Gallay study in \cite{CG23} the case of infinite cylinders with bounded cross sections. 
For the unbounded case considered here, we find that a non-degeneracy condition at infinity is also required.

\end{abstract}

\maketitle

\tableofcontents

\section{Introduction}
In this paper, we study the long-time behavior of a passive scalar advected by a parallel shear flow in an infinite cylinder with unbounded cross section, in arbitrary spatial dimension, focusing on the regime where the viscosity coefficient is small, namely $\nu \ll 1$. To state our results, we fix some notation. Let $\Omega\subset\mathbb{R}^d$
be a smooth domain and let $v:\overline{\Omega}\to\mathbb{R}$ be a smooth function.
We study the evolution of a passive scalar inside the infinite cylinder
\[
  \Sigma \;=\; \mathbb{R}\times\Omega \;\subset\; \mathbb{R}^{d+1},
\]
subject to the shear flow $\vec{u}(y) = \bigl(v(y),0,\cdots,0\bigr)^{\top}$. The scalar density
$f(x,y,t)$ solves the advection–diffusion equation
\begin{equation}\label{eq:feq}
  \partial_t f(x,y,t) + \vec{u}(y)\cdot\,\nabla f(x,y,t)
  \;=\; \nu\,\Delta f(x,y,t)\,, \qquad (x,y)\in\Sigma,\; t>0,
\end{equation}
where $\nu>0$ is the molecular diffusivity and
$\Delta=\partial_x^2+\Delta_y$ is the Laplace operator in all variables $(x,y)\in\Sigma$.
We impose homogeneous Neumann boundary conditions on $\partial\Sigma \;=\; \mathbb{R}\times\partial\Omega$.

Equation~\eqref{eq:feq} is invariant under horizontal translations, so it is natural
to take the partial Fourier transform in $x$:
\begin{equation}\label{eq:fourier}
  \hat f(k,y,t)
  \;=\; \int_{\mathbb{R}} f(x,y,t)\,\mathrm{e}^{-\mathrm{i}kx}\,\mathrm{d}x\,,
  \qquad k\in\mathbb{R},\; y\in\Omega,\; t>0.
\end{equation}
The Fourier mode $\hat f$ satisfies
\begin{equation}\label{eq:feq2}
  \partial_t \hat f(k,y,t) + \mathrm{i}k\,v(y)\,\hat f(k,y,t)
  \;=\; \nu\bigl(-k^2+\Delta_y\bigr)\hat f(k,y,t)\,,
  \qquad y\in\Omega,\; t>0,
\end{equation}
where the horizontal wavenumber $k\in\mathbb{R}$ now acts as a parameter. The purely
horizontal diffusion term $-\nu k^2$ in \eqref{eq:feq2} plays only a minor role in the
regime of interest and can be factored out by setting
\begin{equation}\label{eq:gdef}
  \hat f(k,y,t) \;=\; \mathrm{e}^{-\nu k^2 t}\,g(k,y,t)\,,
  \qquad k\in\mathbb{R},\; y\in\Omega,\; t>0.
\end{equation}
This yields the “hypoelliptic” evolution
\begin{equation}\label{eq:geq}
  \partial_t g(k,y,t) + \mathrm{i}k\,v(y)\,g(k,y,t)
  \;=\; \nu\,\Delta_y g(k,y,t)\,,
  \qquad y\in\Omega,\; t>0,
\end{equation}
which will be the starting point of our analysis. Throughout, we impose homogeneous
Neumann boundary conditions for $g$ on $\partial\Omega$, although the same arguments
apply (with only cosmetic changes) under homogeneous Dirichlet conditions. We also
assume $k\neq 0$, since for $k=0$ equation~\eqref{eq:geq} reduces to the classical
heat equation in $\Omega$.

It is standard that, for any $g_0\in L^2(\Omega)$, equation~\eqref{eq:geq} admits a
unique global solution
\[
  t\mapsto g(k,t)\in C^0\!\bigl([0,\infty),L^2(\Omega)\bigr),
  \qquad g(k,0)=g_0.
\]
(Here and below, $g(k,t)$ denotes the function $y\mapsto g(k,y,t)\in L^2(\Omega)$.)
Our main objective is to quantify the decay rate of solutions to \eqref{eq:geq}
as $t\to\infty$.

We begin with the case where the unbounded cross–section $\Omega$ is one–dimensional, with
$\Omega$ taken to be either $\mathbb{R}$, $(-\infty, L_1)$, or $(L_2, \infty)$.
Our first main result is formulated as follows.

\begin{theorem}[One-dimensional case]\label{thm:main1}
Assume that $d = 1, m\in\mathbb{N}^{*}, v : \overline{\Omega} \to \mathbb{R}$ belongs to $C^m(\overline{\Omega})$ and satisfies the following conditions:  
\begin{enumerate}
    \item[(1)] The derivatives of $v$ up to order $m$ do not vanish simultaneously, namely: \begin{equation}\label{eq:vder} |v'(y)| + |v''(y)| + \cdots + |v^{(m)}(y)| > 0,
\qquad \text{for all } y \in \overline{\Omega}.
\end{equation} 
    \item[(2)] $v$ is non-degenerate at infinity, namely: 
  \begin{equation}\label{non-degenerate at infinity}
\exists\, c_0>0 \quad\text{s.t.} \quad  \varliminf_{|y|\to\infty}|v'(y)|\geq c_0>0.
\end{equation} 
\end{enumerate}
Here, $\Omega$ is taken to be either $\mathbb{R}$, $(-\infty, L_1)$, or $(L_2, \infty)$, where $L_1, L_2$ are arbitrary real numbers. 

Then, there exist positive constants $C$ and $c$ such that, for every $\nu > 0$, 
every $k \neq 0$, and every initial datum $g_0 \in L^2(\Omega)$, the solution 
of \eqref{eq:geq} satisfies, for all $t \ge 0$, 
\begin{equation}\label{lamnuk}
  \|g(k,t)\|_{L^2(\Omega)} \,\le\, 
  C\,\mathrm{e}^{-c \lambda_{\nu,k} t}\,\|g_0\|_{L^2(\Omega)}\,, 
\end{equation}
where 
\begin{equation*}
  \lambda_{\nu,k} \,=\,
  \begin{cases} 
 \nu^{\tfrac{m}{m+2}} |k|^{\tfrac{2}{m+2}}, & \text{if }~ 0 < \nu \le |k|\,, \\[1ex]
  \dfrac{k^2}{\nu}, & \text{if }~ 0 < |k| \le \nu\,.
  \end{cases}
\end{equation*}
\end{theorem}
The case of compact regions has already been treated in detail in \cite{CG23}.
However, for non-compact regions, the non-degeneracy condition at infinity \eqref{non-degenerate at infinity} is required in order to ensure that $\|g(k,t)\|_{L^2(\Omega)}$ admits an exponential decay of \eqref{lamnuk}. In fact, if $v(y)$ does not satisfy the non-degeneracy condition at infinity \eqref{non-degenerate at infinity}, then we can provide the following counterexamples showing that the linearized system \eqref{eq:geq} does not exhibit exponential decay of the form \eqref{lamnuk}.

In our previous work \cite{LZZ-25}, we have shown that if $v(y) = 1/y^{2}$ (a flow commonly referred to by physicists as the Taylor–Couette flow) with $\Omega = [1,\infty)$, then the linear system \eqref{eq:geq} cannot exhibit exponential decay, even in a weighted sense (see \cite{LZZ-25} for details). For this choice of $v(y)$, we have $v^{(m)}(y) > 0$ for all $m \in \mathbb{N}^*$, while $\lim_{y \to \infty} v'(y) = 0$. This example demonstrates that the non-degeneracy condition \eqref{non-degenerate at infinity} is indeed necessary for obtaining exponential decay.

Beyond the one-dimensional case, we also extend the results to higher dimensions.

\begin{theorem}[Higher-dimensional case]\label{thm:higher-d}
Let $d\ge 2$ and $\Omega=\prod_{j=1}^d \Omega_j$, where for each $j\in\{1,\ldots,d\}$ the set
$\Omega_j$ is one of
\[
  (-\infty,\infty),\qquad (-\infty,L_{1,j}),\qquad (L_{2,j},\infty),\qquad (L_{3,j},L_{4,j}),
\]
with arbitrary numbers $L_{1,j},L_{2,j},L_{3,j},L_{4,j}\in\mathbb{R}$ and $L_{3,j}<L_{4,j}$.
Assume the profile is separable:
\[
  v(y)=v_1(y_1)+v_2(y_2)+\cdots+v_d(y_d),
\]
where $y=(y_1,y_2,\ldots,y_d).$ Moreover, for each $j\in\{1,\ldots,d\}$ assume:
\begin{itemize}
  \item[(i)] ($m_j$-th order non-degeneracy) There exists $m_j\in\mathbb{N}^*$ with
  $v_j\in C^{m_j}(\overline{\Omega}_j)$ and
  \[
    \forall\,y_j\in\overline{\Omega}_j,\quad
    |v_j'(y_j)|+|v_j''(y_j)|+\cdots+|v_j^{(m_j)}(y_j)|>0.
  \]
  \item[(ii)] (Non-degeneracy at infinity) If $\Omega_j$ is of the form
  $(-\infty,\infty)$, $(-\infty,L_{1,j})$, or $(L_{2,j},\infty)$, then there exists $c_j>0$ such that
  \[
    \varliminf_{|y_j|\to\infty} |v_j'(y_j)| \;\ge\; c_j.
  \]
\end{itemize}
Then there exist constants $C,c>0$ such that, for any $t\ge 0$,
\[
  \|g(k,t)\|_{L^2(\Omega)}
  \;\le\;
  C\,\exp\Bigl(-\,c\Bigl(\sum_{j=1}^d \lambda_{\nu,k,j}\Bigr)t\Bigr)\,\|g_0\|_{L^2(\Omega)},
\]
where, for each $j\in\{1,\ldots,d\}$,
\[
  \lambda_{\nu,k,j} \,=\,
  \begin{cases}
    \nu^{\frac{m_j}{m_j+2}}\,|k|^{\frac{2}{m_j+2}}, & \text{if } 0<\nu \le |k|,\\[1ex]
    \dfrac{k^2}{\nu}, & \text{if } 0< |k| \le \nu.
  \end{cases}
\]
\end{theorem}

It is worth noting that in \cite{CG23}, the authors also discussed the higher-dimensional case where the cross-section $\Omega$ is a bounded domain. They mainly focused on the situation that $v(y)$ is a smooth Morse function without critical points on the boundary $\partial\Omega$. This means that $v(y)$ has finitely many critical points, all of which are nondegenerate, and that in a neighborhood of each critical point, $v(y)$ is locally diffeomorphic to a quadratic polynomial.  The finiteness of the critical points plays a crucial role in their approach, and therefore the method cannot be extended to unbounded domains, where this finiteness condition fails.

The main approach of this paper is to use resolvent estimates to obtain our main results. Recall that $\Omega \subset \mathbb{R}^d$ is a smooth domain, and that $v : \overline{\Omega} \to \mathbb{R}$ is a smooth function.
For any $\nu > 0$ and any $k \neq 0$, the linear evolution equation \eqref{eq:geq} can be written in the abstract form
\begin{equation}\label{eq:Hdef}
\partial_t g + H_{\nu,k} g = 0,
\qquad \text{where} \qquad
H_{\nu,k} = -\nu \Delta_y + \mathrm{i}k v(y).
\end{equation}
We regard $H_{\nu,k}$ as a linear operator on the Hilbert space $X = L^2(\Omega)$ with domain
$$
  D(H_{\nu,k}) \,=\, 
  \Bigl\{ g \in H^2(\Omega)\,;\, \mathcal{N} \cdot \nabla g = 0 
  \;\text{on } \partial \Omega \Bigr\},
$$
where $\mathcal{N}$ denotes the outward unit normal vector on $\partial\Omega$. Moreover, for all $g \in D(H_{\nu,k})$, one has the identities
\begin{equation}\label{numrange}
  \Re\,\langle H_{\nu,k} g, g\rangle \,=\, \nu \|\nabla_y g\|^2 \,\ge\, 0\,,
  \qquad \text{and}\qquad 
  \Im\,\langle H_{\nu,k} g, g\rangle \,=\, k \int_\Omega v(y)|g(y)|^2 \,\mathrm{d}y\,.
\end{equation}
Here and throughout, we denote by $\langle\cdot,\cdot\rangle$ the inner product on $X=L^2(\Omega)$, and by $\|\cdot\|$ the associated norm.

Our main objective is to obtain sharp estimates for the resolvent norm $\|(H_{\nu,k} - z)^{-1}\|_{X\to X}$ uniformly for all $z \in \mathrm{i}\mathbb{R}$. 
Equivalently, this amounts to deriving a precise lower bound on the pseudospectral 
abscissa $\Psi(H_{\nu,k})$ introduced below.
\begin{equation}\label{eq:Psidef}
  \Psi(H_{\nu,k}) \,:\,=\, 
  \biggl(\sup_{z \in \mathrm{i}\mathbb{R}} \|(H_{\nu,k} - z)^{-1} \|_{X\to X}\biggr)^{-1},
\end{equation}
as a function of the parameters $\nu$ and $k$. 

This quantity is first introduced and studied by Gallagher, Gallay and Nier in \cite{GGN}. 
A recent result of Wei \cite{W21} (see also Helffer and Sj\"ostrand \cite{HS21}) 
shows that the pseudospectral abscissa $\Psi(H_{\nu,k})$ provides a decay estimate for the 
semigroup generated by $-H_{\nu,k}$. In particular, by applying \cite[Theorem~1.3]{W21}, 
we obtain
\begin{lemma}[\cite{W21}]\label{GP lemma}
Let $H_{\nu,k}$ be a $m$-accretive operator in a Hilbert space $X$, then the following estimates hold:
\begin{equation}
  \, \| \mathrm{e}^{-tH_{\nu,k}} \|_{X\to X}
   \,\leq\, \mathrm{e}^{-t\Psi(H_{\nu,k}) + \tfrac{\pi}{2}}, 
   \qquad \text{for all } t \geq 0.
\end{equation}
\end{lemma}

It is worth noting that Lemma \ref{GP lemma} provides a very effective method to characterize the long-time decay behavior of linear evolution equations. However, one should also keep in mind that the optimal pseudospectral estimate is not necessarily equivalent to the optimal decay estimate for the evolution equation.    

For instance, in \cite{LWZ20a}, it has been proved that for the linearized Oseen vortex operator $L_\alpha$, the pseudospectral bound $ \Psi(L_\alpha)$ is $|\alpha|^{\f13}(|\alpha|\gg1)$, while the spectral bound $\Phi(L_\alpha)$ is $|\alpha|^{\f12}$, that is
\begin{align*}
C_1|\alpha|^{\f13}\leq\Psi(L_\alpha)\le C_2|\alpha|^{\f13},\quad \Phi(L_\alpha):=\inf_{\lambda\in\sigma(L_\alpha)}\Re\lambda\geq C|\alpha|^{\f12},
\end{align*}
here $C, C_1, C_2$ denote positive constants independent of $\alpha$, and $\sigma(L_\alpha)$ denotes the spectrum of the operator $L_\alpha$. Consequently, from the pseudospectral and spectral bounds one can deduce 
\begin{align}
\label{decay rate1}   
&    \| \mathrm{e}^{-tL_\alpha} \|_{X\to X}
   \,\leq\, C\mathrm{e}^{-c|\alpha|^{\f13}t},\\
\label{decay rate2}   &  \varlimsup_{t\to\infty} \f{1}{t}\log\| \mathrm{e}^{-tL_\alpha} \|_{X\to X}
   \,\leq\, -c|\alpha|^{\f12},
\end{align}
here $C, c$ denote positive constants independent of $\alpha$ and the space $X$ refers to the weighted $L^2$ space with Gaussian weight, defined after performing the self-similar transformation. In our recent work \cite{LZZ-Prep}, we also obtain that the lower bound of estimate \eqref{decay rate2}  is $-|\alpha|^{\f12}$, namely:
\begin{align*}    \varlimsup_{t\to\infty} \f{1}{t}\log\| \mathrm{e}^{-tL_\alpha} \|_{X\to X}
   \,\geq\, -c’|\alpha|^{\f12}.
\end{align*}
It is worth noting that \eqref{decay rate2} only provides a decay estimate of the form 
\begin{align}
   \label{decay rate3}  \| \mathrm{e}^{-tL_\alpha} \|_{X\to X}
   \,\leq\, \tilde{C}(\alpha)\mathrm{e}^{-c|\alpha|^{\f12}t},
\end{align}
where $\tilde{C}$ depends on $\alpha$ and cannot be given in an explicit expression. Moreover, the fact that $\tilde C$ cannot be chosen independently of $\alpha$ is due to the non-selfadjointness of $L_\alpha$.

But it must be said that Lemma \ref{GP lemma} provides a very clean and convenient way to transform semigroup decay estimates into resolvent estimates for the operator. Therefore, the key computation in this section is to obtain an estimate for the pseudospectral bound $\Psi(H_{\nu,k})$. Hence, the proof of Theorem \ref{thm:main1} reduces to establishing the following proposition.

\begin{proposition}\label{Enhanced dissipation1d}Assume $d=1$ and $v : \Omega\to \mathbb{R}$ is a $C^{m}$ function, where $m$ and $\Omega$ are the same as in Theorem \ref{thm:main1}, and $v$ satisfies conditions \eqref{eq:vder} and \eqref{non-degenerate at infinity} of Theorem \ref{thm:main1}. Then
there exists a constant $C > 0$ such that, for all $\nu > 0$ and all $k \neq 0$,
\begin{equation}\label{eq:resol main}
  \Psi(H_{\nu,k})  \,\ge\, C\,\lambda_{\nu,k}\,,
\end{equation}
where $\Psi(H_{\nu,k}) $ is defined in \eqref{eq:Psidef}. 
\end{proposition}

To prove the lower bound \eqref{eq:resol main}, we first need to introduce some notation. Write any $z\in \mathrm{i}\mathbb{R}$ as $z=\mathrm{i} k\,\lambda$ with $\lambda\in\mathbb{R}$,  then
\begin{equation}
H_{\nu,k}-z \;=\; H_{\nu,k,\lambda} \;:=\; -\nu\,\Delta_y \;+\; \mathrm{i} k\bigl(v(y)-\lambda\bigr).\label{eq:Hdef2}
\end{equation}
The derivation of the resolvent estimate \eqref{eq:resol main} relies essentially on the finiteness of the following level sets. Let
\begin{equation}
 \label{level set-1} E_\lambda := \bigl\{ y \in \Omega \;\big|\; v(y) = \lambda \bigr\}, 
  \qquad \forall\,\lambda \in \mathbb{R}.
\end{equation}
This is precisely why we need to discuss the finiteness of the level set: because if $y \in \{y\in\Omega:|v(y)-\lambda|\geq\delta^m\}$, then in a certain sense, the imaginary part $\Im< H_{\nu,k,\lambda}g,g>$ provides a coercive estimate of the form $\delta^m\int_{\{y\in\Omega:|v(y)-\lambda|\geq\delta^m\}}|g|^2dy$. However, for the integral over $
\{ y \in \Omega : |v(y) - \lambda| <\delta^m \}$, we can only hope to control it by using the dissipative term $-\nu\,\partial_y^2$. Moreover, the measure of $\{ y \in \Omega : |v(y) - \lambda| <\delta^m \}$ should not be too large. 

Thus, without loss of generality, we define the so-called thickened level set as follows, following the notation in \cite{CG23}:
\begin{equation}\label{eq:Edef}
  E_{\lambda,\delta}^m :\,=\, \bigl\{y \in \Omega \,\big|\, |v(y) - \lambda| 
  < \delta^m\bigr\}\,,
\end{equation}
the corresponding $\delta$-neighborhood of $E_{\lambda,\delta}^m$ is defined by
\begin{equation}\label{eq:Ndelta}
   \mathcal{E}_{\lambda,\delta}^m
  :\,= \,\bigl\{ y \in \Omega \,\big|\, \operatorname{dist}(y, E_{\lambda,\delta}^m ) < \delta \bigr\}.
\end{equation}

In \cite{CG23}, for the higher-dimensional case, the authors also directly treated neighborhoods of this type around the set $\mathcal{E}_{\lambda,\delta}^m$. 
However, in higher dimensions, such a definition of distance may be effective only when $v(y)$ is a homogeneous polynomial, for example, when $v(y)$ is a Morse function.

In fact, the approach of decomposing the domain into $\mathcal{E}_{\lambda,\delta}^m$ and $\Omega\setminus\mathcal{E}_{\lambda,\delta}^m$ to derive resolvent estimates \eqref{eq:resol main} can already be found in \cite{AHL-2,CDLZ25,CLWZ-2D-C,CWZ24,DL22,GGN,LWZ20} and subsequent literature.

\begin{proof}[\bf Proof of Proposition \ref{Enhanced dissipation1d}]

For notational convenience, we set $H:=H_{\nu,k,\lambda}$, $E:=E_{\lambda,\delta}^m$ and $ \mathcal{E}:=\mathcal{E}_{\lambda,\delta}^m$. Since our goal is to use the imaginary part $\Im< Hg,g>$ to derive the coercive estimate $\delta^m\int_{\Omega\setminus\mathcal{E}}|g|^2dy$ in $\Omega\setminus\mathcal{E}$\ forall $g\in D(H_{\nu,k})$ and $\delta\in(0,\delta_0],$ where $\delta_0$ is an fixed number less than one to be determined later, and since $v(y)-\lambda$ may change sign, it is natural to introduce a cutoff function $\chi$ such that $\chi\bigl(v(y)-\lambda\bigr)=\bigl|v(y)-\lambda\bigr|$ on $\Omega\setminus\mathcal{E}$.

In fact, for different shear flows, such as the Couette flow $v(y)=y$ \cite{CLWZ-2D-C,CWZ24}, the Kolmogorov flow $v(y)=\sin y$ \cite{LWZ20}, and the Poiseuille flow $v(y)=y^{2}$ \cite{CDLZ25,DL22}, different cutoff functions $\chi$ have been devised so that $\chi\bigl(v(y)-\lambda\bigr)=\bigl|v(y)-\lambda\bigr|$ on $\Omega\setminus\mathcal{E}$ holds. 

In \cite{CG23}, a general cutoff function $\chi$ is constructed so that the approach applies to any $C^{m}$ function $v(y)$. For convenience, we also use the cutoff function $\chi$ introduced in \cite{CG23}, which is defined as follows:
\[
  \chi(y) := \phi\!\left(\frac{1}{\delta}\,
  \operatorname{sign}\bigl(v(y)-\lambda\bigr)\,
  \operatorname{dist}\bigl(y,E\bigr)\right), 
  \qquad y \in \Omega,
\]
where $\phi : \mathbb{R} \to [-1,1]$ is the unique odd function satisfying 
$\phi(t) = \min\{t,1\}$ for $t \ge 0$. 
The cutoff function $\chi$ has the following properties:
\begin{itemize}
    \item[(i)] $\|\chi\|_{L^\infty} \le 1$, 
$\|\nabla_y \chi\|_{L^\infty} \le 1/\delta$;
 \item[(ii)] $\chi(y)\,\bigl(v(y)-\lambda\bigr) \ge 0$ for all $y \in \Omega$;
 \item[(iii)] If $y \in \Omega \setminus \mathcal{E}$, then $\chi(y) = \operatorname{sign}\bigl(v(y)-\lambda\bigr)$.
\end{itemize}
A direct calculation yields
$$
\Im \langle Hg,\chi g\rangle \;=\; 
\nu\,\Im\langle \nabla_y g, g \nabla_y\chi\rangle \;+\; k \langle \chi(v-\lambda)g, g\rangle,
$$
since $\|\chi\|_{L^\infty} \le 1$, 
$\|\nabla_y \chi\|_{L^\infty} \le 1/\delta$, it follows that
\begin{equation}\label{imaginary part}
  |k| \langle \chi(v-\lambda)g,g\rangle \,\le\, \|Hg\|\,\|g\| + 
\nu/\delta\,\|\nabla_y g\|\,\|g\|.
\end{equation}
Note that the operator $H$ is accretive, and that $\nabla_y g$ satisfies the following upper bound:
\begin{equation}\label{H is accretive}
  \Re \langle Hg,g\rangle \,=\, \nu \|\nabla_y g\|^2\, \Longrightarrow 
 \nu \|\nabla_y g\|^2 \,\le\, \|Hg\|\,\|g\|\,.
\end{equation}
 Hence, \eqref{H is accretive}, together with \eqref{imaginary part}, implies
\begin{equation*}
\begin{split}
|k|\delta^m \int_{\Omega\setminus \mathcal{E}} |g(y)|^2\,\mathrm{d}y \,&\le\,
  |k|\int_{\Omega\setminus \mathcal{E}} \chi(y)\bigl(v(y)-\lambda\bigr) |g(y)|^2\,\mathrm{d}y \,\le\, 
|k| \langle \chi(v-\lambda)g,g\rangle \\ 
  \,&\le\, \|Hg\|\,\|g\| + 
\nu/\delta\,\|\nabla_y g\|\,\|g\|
 \,\le\, 
\|Hg\|\,\|g\| + \nu^{1/2}/\delta\,\|Hg\|^{1/2}\,\|g\|^{3/2},
\end{split}
\end{equation*}
and this yields
\begin{equation}
\label{outside}\begin{split}
 \int_{\Omega\setminus \mathcal{E}} |g(y)|^2\,\mathrm{d}y 
 \,\le\, 
|k|^{-1}\delta^{-m}\|Hg\|\,\|g\| + \nu^{1/2}|k|^{-1}\delta^{-m-1}\,\|Hg\|^{1/2}\,\|g\|^{3/2}.
\end{split}
\end{equation}

For the integral over the interior region $\mathcal{E}$, combining with Lemma \ref{gx}, \eqref{H is accretive}, we have:
\begin{equation}
\label{inside}\int_{\mathcal{E}} |g(y)|^2\,\mathrm{d}y \,\leq\, m(\mathcal{E})\,\|g\|_{L^\infty}^2\,\leq\, 2m(\mathcal{E})\,\|\nabla_y g\|\,\|g\|
 \,\le\, 2\nu^{-\f12}m(\mathcal{E})\,\|Hg\|^{1/2}\,\|g\|^{3/2},
\end{equation}
where $m(\mathcal{E})$ denotes the Lebesgue measure of the set $\mathcal{E}$.

Consequently, using \eqref{outside} and \eqref{inside} together, we deduce that
\begin{equation}
\label{full}\|g\|^2 \,\lesssim\, \max\{|k|^{-1}\delta^{-m},\nu|k|^{-2}\delta^{-2m-2},\nu^{-1}m^2(\mathcal{E})\} \,\|Hg\|\,\|g\|.
\end{equation}
Thus, choosing $\delta^{m+2} \sim\nu/|k|$ yields $|k|^{-1}\delta^{-m}\sim\nu|k|^{-2}\delta^{-2m-2}$. Requiring 
\begin{align*}
    |k|^{-1}\delta^{-m}\sim\nu|k|^{-2}\delta^{-2m-2} \sim\nu^{-1}m^2(\mathcal{E})
\end{align*}
further imposes that $m(\mathcal{E})$ should be of order $\delta$. Hence, if the set $\mathcal{E}$ has measure of order $\delta$, it follows from \eqref{full} that
\begin{equation}
  \label{full resolvent}  \nu^{\tfrac{m}{m+2}} |k|^{\tfrac{2}{m+2}}\|g\|\lesssim \|Hg\|.
\end{equation}

However, note that in the subsequent proof, in Lemma~\ref{|E(lambda)|}, Lemma \ref{mathcal{E}(lambda)} and Proposition \ref{mathcal e measure-prop}, the smallness of~$\delta$ is essential.

For the case $0<\nu\leq|k|,$ we choose
    $\delta =\de_0\nu^{\f{1}{m+2}}|k|^{-\f{1}{m+2}}\leq\delta_0$, where $\delta_0$ to be defined in Proposition \ref{mathcal e measure-prop} is a constant independent of $\nu$ and $k$. This means that the gap of {\it enhanced dissipation} type \eqref{full resolvent} only appears for the frequency range $|k| \ge \nu$. That is, 
\begin{equation}
  \label{enhanced dissipation type}  \nu^{\tfrac{m}{m+2}} |k|^{\tfrac{2}{m+2}}\|g\|\lesssim \|Hg\|,\quad \text{ for all } |k|\geq\nu.
\end{equation}

For the case $0<|k| \le \nu$, we may choose $\delta = \delta_0$, and it suffices to show that $m(\mathcal{E})$ is finite. In this case, the lower bound is referred to as {\it{Taylor dispersion}}, namely 
\begin{equation}
  \label{Taylor dispersion}  k^2/\nu\|g\|\lesssim \|Hg\|,\quad \text{ for all } |k|\leq\nu.
\end{equation}

At this point, completing the proof of Proposition \ref{Enhanced dissipation1d} amounts to establishing the estimate 
\begin{align*}
  m(\mathcal{E}) \leq C\delta,
\end{align*}
where 
\begin{equation*}
  \delta \,=\,
  \begin{cases} 
 \delta_0\left(\f{\nu}{|k|}\right)^\f{1}{m+2}, & \text{if }~ 0 < \nu \le |k|\,, \\[1ex]
  \delta_0, & \text{if }~ 0 < |k| \le \nu\,.
  \end{cases}
\end{equation*}
See Proposition \ref{mathcal e measure-prop} for details.
\end{proof}

It should be noted that the proof of Proposition \ref{Enhanced dissipation1d} does not rely on the non-degeneracy assumptions \eqref{eq:vder}, \eqref{non-degenerate at infinity}. These assumptions are introduced to ensure $
m(\mathcal{E}_{\lambda,\delta}^m) \leq C\delta$. For details, see Section~\ref{sec3}.

In fact, for $v(y)$ on a compact region, regardless of whether the non-degeneracy condition \eqref{eq:vder} holds, the dissipativity of $-\nu \Delta_y$ alone is sufficient to guarantee that the linear system \eqref{eq:Hdef} decays exponentially. The role of assumption \eqref{eq:vder} is to determine the precise enhanced dissipation rate $\exp(-c\nu^{\tfrac{m}{m+2}} |k|^{\tfrac{2}{m+2}}t)$ of the system \eqref{eq:Hdef}, rather than merely to ensure exponential decay.

In this paper, however, we consider the long-time behavior of solutions in non-compact regions. As mentioned above, assumption \eqref{non-degenerate at infinity} is imposed to guarantee that the linearized system \eqref{eq:Hdef} still exhibits exponential decay in non-compact regions. In \cite{LZZ-25}, for the Taylor–Couette flow $v(y)=1/y^{2}$ in the exterior region $\Omega = [1,\infty)$, we have already shown that \eqref{eq:Hdef} cannot exhibit exponential decay. In this case, we obtain
 \begin{align}
    m(\mathcal{E}_{0,\delta}^m) = \infty \text{ for any }m\in\mathbb{N}^{*}.  
 \end{align}

This example of the Taylor–Couette flow in the exterior domain illustrates that, for a non-compact domain $\Omega$, the non-degeneracy condition at infinity is not only sufficient to guarantee $m(\mathcal{E}_{\lambda,\delta}^m) \leq C\delta$, but is also a necessary condition for the linear operator to possess a positive pseudospectral gap  $\Psi(H_{\nu,k})$ — since, according to Lemma \ref{GP lemma}, a positive pseudospectral gap directly implies exponential decay. 

\section{Uniform-in-$\lambda$ estimate for $m(\mathcal{E}_{\lambda,\delta}^m)$ in compact regions}\label{sec2}

Prior to the proof of Proposition \ref{mathcal e measure-prop}, we introduce several auxiliary lemmas that will be used in the argument.

\begin{lemma}\label{piecewise strict monotonicity}(Piecewise strict monotonicity)
Let $d=1$, $m \geq 1$, and $L_1, L_2$ are arbitrary real numbers with $L_1 < L_2$. Suppose $v:[L_1,L_2] \to \mathbb{R}$ is a $C^m$ function satisfying the $m$-th order non-degeneracy condition
\[
 |v'(y)| + |v''(y)| + \cdots + |v^{(m)}(y)| > 0,\quad \forall\, y \in [L_1,L_2].
\]
Then $v$ is piecewise strictly monotone; that is, there exist $N \in \mathbb{N}^*$ and points 
\[
y_0,y_1,\ldots,y_N \in [L_1,L_2], \qquad L_1 = y_0 < y_1 < \cdots < y_N = L_2,
\]
such that $v(y)$ is strictly monotone on each interval $[y_{j-1},y_j]$ for $j=1,\ldots,N$, and the monotonicity direction of $v$ alternates between two consecutive intervals.

\end{lemma}
\begin{proof}
\begin{itemize}
  \item[(1)]\textbf{Case $m=1$.}
  In this case, we have
  \begin{equation}\label{yj}
  \forall\, y \in [L_1,L_2],\quad |v'(y)| > 0,
  \end{equation}
  which implies that either 
  \begin{equation}\label{zf}
 v'(y) > 0\,(\forall\, y \in [L_1,L_2])\text{ or } v'(y) < 0\, (\forall\, y \in [L_1,L_2]).
  \end{equation}
Otherwise, there exist $\xi,\eta\in[L_1,L_2]$ such that $v'(\xi)<0$ and $v'(\eta)>0$.  
Without loss of generality, assume $\eta<\xi$. Since $v'$ is continuous, by the intermediate value theorem there exists $\zeta\in(\eta,\xi)$ such that $v'(\zeta)=0$, which contradicts \eqref{yj}.  
Therefore,  \eqref{zf} holds, corresponding respectively to $v$ being strictly increasing or strictly decreasing. In this case, take $N=1$, $y_0=L_1$, and $y_1=L_2$, and the case $m=1$ is proved.
\item[(2)]\textbf{Case $m\geq2$.} We now analyze the case $m\geq 2$. 

Let $y_0:=L_1,  k_0:=\min\{1\leq i\leq m:\ v^{(i)}(y_0)\neq 0\}.$ By the $m$-th order non-degeneracy condition on $v$, $k_0$ is well-defined. Without loss of generality, assume $v^{(k_0)}(y_0)>0$.  
Expanding $v'(y)$ at $y=y_0$ using Taylor’s theorem (in particular, when $k_0=1$ this reduces to the definition of continuity), we obtain
\[
v'(y)=\frac{v^{(k_0)}(y_0)}{(k_0-1)!}(y-y_0)^{k_0-1}+o\bigl(|y-y_0|^{k_0-1}\bigr),
\]
hence there exists $\delta_0'\in(0,L_2-y_0]$ such that, for any $y\in[y_0,y_0+\delta_0']$,
\[
\left|\,v'(y)-\frac{v^{(k_0)}(y_0)}{(k_0-1)!}(y-y_0)^{k_0-1}\right|
\leq \frac{v^{(k_0)}(y_0)}{2(k_0-1)!}(y-y_0)^{k_0-1},
\]
and therefore
\[
\frac{v^{(k_0)}(y_0)}{2(k_0-1)!}(y-y_0)^{k_0-1}\ \leq\ v'(y)\ \leq\
\frac{3v^{(k_0)}(y_0)}{2(k_0-1)!}(y-y_0)^{k_0-1}.
\]
Consequently, $v'(y)>0$ on $(y_0,y_0+\delta_0']$, and since $v$ is continuous at $y=y_0$, it follows that $v$ is strictly increasing on $[y_0,y_0+\delta_0']$.

Define $y_1:=\sup\{y\in(y_0,L_2]:\ v \text{ is strictly monotone on }[y_0,y]\}.$ Then $y_1$ is well-defined and $y_1\in[y_0+\delta_0',L_2]$. By the definition of the supremum, there exists an increasing sequence $\{z_k\}$ such that $v$ is strictly increasing on $[y_0,z_k]$ and
\begin{equation}\label{zk}
 \lim_{k\to\infty}z_k=y_1.
\end{equation}
We now prove $v$ is strictly increasing on $[y_0,y_1]$. Let $w_1,w_2\in[y_0,y_1]$ with $w_1<w_2$.  

\emph{Case 1:} $w_2<y_1$. Then there exists $k\in\mathbb{N}^*$ such that $w_2<z_k$. Hence $w_1,w_2\in[y_0,z_k]$, and since $v$ is strictly increasing on $[y_0,z_k]$, we have $v(w_1)<v(w_2)$.

\emph{Case 2:} $w_2=y_1$. Then there exists $k'\in\mathbb{N}^*$ such that for all $k\geq k'$, one has $w_1<z_k$. Since $v$ is strictly increasing on $[y_0,z_k]$, it follows that
\begin{equation}\label{sldd}
\forall k\geq k',\quad v(w_1)< v(z_k).
\end{equation}
Moreover, as $\{z_k\}$ is increasing, we have $z_k\in[y_0,z_{k+1}]$. Since $v$ is increasing on $[y_0,z_{k+1}]$, it follows that $v(z_k)\leq v(z_{k+1})$, so $\{v(z_k)\}$ is increasing. In particular,
\begin{equation}\label{kk}
\forall k\geq k',\quad v(z_{k'})\leq v(z_k).
\end{equation}
By the continuity of $v$,
\[
v(w_1)\overset{(\ref{sldd})}{<}v(z_{k'})\overset{(\ref{kk})}{\leq}\lim_{k\to\infty}v(z_k)
= v\!\left(\lim_{k\to\infty}z_k\right)\overset{(\ref{zk})}{=}v(y_1)=v(w_2).
\]
In conclusion, $v$ is strictly increasing on $[y_0,y_1]$.

If $y_1=L_2$, then $v$ is strictly increasing on $[y_0,y_1]=[L_1,L_2]$, and taking $N=1$ completes the proof.

Otherwise, $y_1\in[y_0+\delta_0',L_2)$, and we have $v'(y_1)=0$. Indeed, since $v$ is strictly increasing on $[y_0,y_1]$, we have $v'(y)\geq0$ for every $y\in(y_0,y_1)$. Letting $y\ \uparrow\ y_1$ and using the continuity of $v'$, we obtain
\begin{equation}\label{bdh}
 v'(y_1)\geq 0.
\end{equation}
If the inequality in \eqref{bdh} were strict, i.e. $v'(y_1)>0$, then by continuity of $v'$ there would exist $\delta_0''\in(0,L_2-y_1]$ such that $v'(y)>0$ for all $y\in[y_1,y_1+\delta_0'']$. Hence $v$ would be strictly increasing on $[y_1,y_1+\delta_0'']$, and combined with strict monotonicity on $[y_0,y_1]$ we would conclude that $v$ is strictly increasing on $[y_0,y_1+\delta_0'']$, contradicting the definition of $y_1$. Therefore,
\begin{equation}\label{dyl}
v'(y_1)=0.
\end{equation}

Let $k_1=\min\{1\leq i\leq m:\ v^{(i)}(y_1)\neq0\}.$ By the $m$-th order non-degeneracy condition, $k_1$ is well-defined and $v^{(k_1)}(y_1)\neq0$. From \eqref{dyl} we have $v'(y_1)=0$, hence $k_1>1$. We claim that $v^{(k_1)}(y_1)<0$.  
Otherwise, if $v^{(k_1)}(y_1)>0$, then Taylor’s formula gives
\begin{equation}\label{tl}
v'(y)=\frac{v^{(k_1)}(y_1)}{(k_1-1)!}(y-y_1)^{k_1-1}+o\bigl(|y-y_1|^{k_1-1}\bigr),
\end{equation}
so there exists $\delta_0'''\in(0,L_2-y_1]$ such that, for any $y\in[y_1,y_1+\delta_0''']$,
\[
\left|\,v'(y)-\frac{v^{(k_1)}(y_1)}{(k_1-1)!}(y-y_1)^{k_1-1}\right|
\leq \frac{v^{(k_1)}(y_1)}{2(k_1-1)!}(y-y_1)^{k_1-1},
\]
and hence
\[
\frac{v^{(k_1)}(y_1)}{2(k_1-1)!}(y-y_1)^{k_1-1}
\ \leq\ v'(y)\ \leq\
\frac{3v^{(k_1)}(y_1)}{2(k_1-1)!}(y-y_1)^{k_1-1}.
\]
Thus $v'(y)>0$ for all $y\in(y_1,y_1+\delta_0''']$, and by continuity at $y_1$ the function $v$ is strictly increasing on $[y_1,y_1+\delta_0''']$. Combined with the strict increase on $[y_0,y_1]$, this implies that $v$ is strictly increasing on $[y_0,y_1+\delta_0''']$, contradicting the definition of $y_1$. Hence $v^{(k_1)}(y_1)<0$.

Again by \eqref{tl}, there exists $\delta_1'\in(0,L_2-y_1]$ such that, for any $y\in(y_1,y_1+\delta_1']$,
\[
\left|\,v'(y)-\frac{v^{(k_1)}(y_1)}{(k_1-1)!}(y-y_1)^{k_1-1}\right|
\leq -\,\frac{v^{(k_1)}(y_1)}{2(k_1-1)!}(y-y_1)^{k_1-1},
\]
whence
\[
\frac{3v^{(k_1)}(y_1)}{2(k_1-1)!}(y-y_1)^{k_1-1}
\ \leq\ v'(y)\ \leq\
\frac{v^{(k_1)}(y_1)}{2(k_1-1)!}(y-y_1)^{k_1-1},
\]
so $v'(y)<0$ on $(y_1,y_1+\delta_1']$. By continuity of $v$ at $y_1$, $v$ is strictly decreasing on $[y_1,y_1+\delta_1']$.

Define
\[
y_2=\sup\Bigl\{y\in(y_1,L_2]:\ v \text{ is strictly decreasing on }[y_1,y]\Bigr\}.
\]
By a similar argument, $y_2$ is well-defined and $v$ is strictly decreasing on $[y_1,y_2]$. In this way, the monotonicity of $v$ differs on the adjacent intervals $[y_0,y_1]$ and $[y_1,y_2]$.

If $y_2=L_2$, take $N=2$ and the claim follows. Otherwise $v'(y_2)=0$, and the above procedure can be iterated further. We assert that this process terminates in finitely many steps; that is, there exists $N'\in\mathbb{N}^*$ such that $y_{N'}=L_2$.

Otherwise, the above construction produces an infinite strictly increasing sequence $\{y_i\}\subset[L_1,L_2]$ with $v'(y_i)=0$. By the monotone convergence theorem, there exists $\bar{y}\in[L_1,L_2]$ such that
\[
\lim_{i\to\infty}y_i=\bar{y}.
\]
Therefore, by continuity of $v'$,
\[
v'(\bar{y})=\lim_{i\to\infty}v'(y_i)=0.
\]
Since $v'(y_i)=v'(y_{i+1})=0$, Rolle’s theorem implies the existence of $r_i\in(y_i,y_{i+1})$ such that $v''(r_i)=0$. Hence, for all $i\in\mathbb{N}^*$,
\[
y_i<r_i<y_{i+1}<r_{i+1},
\]
so $\{r_i\}$ is a strictly increasing sequence and, by the squeeze theorem,
\[
\lim_{i\to\infty}r_i=\bar{y}.
\]
By continuity of $v''$,
\[
v''(\bar{y})=\lim_{i\to\infty}v''(r_i)=0.
\]

Continuing in this way, we conclude that $v^{(i)}(\bar{y})=0$ for every $i\in\{1,\ldots,m\}$, which contradicts the $m$-th order non-degeneracy condition on $v$. This completes the proof.
\end{itemize}
\end{proof}

By the piecewise monotonicity property given in Lemma \ref{piecewise strict monotonicity}, we obtain that the number of elements of the level set of $v$ admits a uniform bound on any bounded region.

\begin{lemma}\label{Uniformly finitely many intersection points}
Let $d=1$, $m \geq 1$, and $L_1, L_2$ are arbitrary real numbers with $L_1 < L_2$. Suppose $v:[L_1,L_2] \to \mathbb{R}$ is a $C^m$ function satisfying the $m$-th order non-degeneracy condition
\[
 |v'(y)| + |v''(y)| + \cdots + |v^{(m)}(y)| > 0,\quad \forall\, y \in [L_1,L_2].
\]
Then the number of elements of the level set of $v$ admits a uniform bound on $[L_1,L_2]$, namely, there exists $N_0 \in \mathbb{N}^*$ such that, for every $\lambda \in \mathbb{R}$,
$$
\mathrm{Card }(E_\lambda) \leq N_0.
$$
\end{lemma}
\begin{proof}
    By Lemma \ref{piecewise strict monotonicity}, we know that $v$ is piecewise strictly monotone. Hence, for any $\lambda \in \mathbb{R}$, the equation $v(y)=\lambda$ has at most one intersection point in each of the intervals $[y_0,y_1], [y_1,y_2], \ldots, [y_{N-1},y_N]$. Therefore,
\[\operatorname{card}(E_\lambda) \leq N. \]
By choosing $N_0=N,$ we complete the proof.
\end{proof}

In addition, we also need the estimate of $m(E^m_{\lambda',\delta})$ as a preliminary step. Since the proof of Lemma~\ref{prior estimate for E(lambda)} requires a clear case-by-case analysis, the argument is somewhat lengthy.

\begin{lemma}\label{prior estimate for E(lambda)}
Let $d=1$, $m \geq 1$, and $L_1, L_2$ are arbitrary real numbers with $L_1 < L_2$. Suppose $v:[L_1,L_2] \to \mathbb{R}$ is a $C^m$ function satisfying the $m$-th order non-degeneracy condition
\[
 |v'(y)| + |v''(y)| + \cdots + |v^{(m)}(y)| > 0,\quad \forall\, y \in [L_1,L_2].
\]
Denote $v_1 = \min_{y \in [L_1,L_2]} v(y), \ 
v_2 = \max_{y \in [L_1,L_2]} v(y)$, then $v_1 < v_2$, and the following holds:
  for all $\lambda \in [v_1,v_2]$, there exist $\delta_0(\lambda) > 0$ and $C(\lambda) > 0$ such that, for any 
$\lambda' \in \bigl(\lambda - (\delta_0(\lambda))^m, \, \lambda + (\delta_0(\lambda))^m\bigr)$ and any 
$\delta \in (0, \delta_0(\lambda)]$, we have
\[
m(E^m_{\lambda',\delta}) \;\leq\; C(\lambda)\,\delta.
\]
\end{lemma}
\begin{proof}From the definition of $v_1$ and $v_2$ it follows that $v_1 \leq v_2$.  
If equality holds, i.e. $v_1 = v_2$, then $v$ is a constant function on $\Omega$, and all its derivatives vanish, which contradicts the $m$-th order non-degeneracy condition on $v$.  
Hence, $v_1 < v_2$.

For any $\lambda \in [v_1,v_2]$, by the intermediate value theorem for continuous functions, there exists some $\xi \in [L_1,L_2]$ such that $v(\xi) = \lambda$.  
Thus,
\[
E_\lambda = \{ y \in [L_1,L_2] : v(y) = \lambda \} \neq \varnothing.
\]

By Lemma \ref{Uniformly finitely many intersection points}, we know that $E_\lambda$ is a finite set, and $1 \leq N(\lambda) := \operatorname{card}(E_\lambda) \leq N_0$, where $N_0$ is given in Lemma \ref{Uniformly finitely many intersection points}. Without loss of generality, we may write
\begin{equation}\label{el}
E_\lambda = \{\tilde y_1,\ldots,\tilde y_{N(\lambda)}\}, 
\qquad L_1 \leq \tilde y_1< \cdots < \tilde y_{N(\lambda)} \leq L_2.
\end{equation}
For any $i \in \{1,\ldots, N(\lambda)\}$, define
\begin{equation}\label{ki}
\tilde k_i = \min \Bigl\{\, 1 \leq j \leq m : v^{(j)}(\tilde y_i) \neq 0 \Bigr\}.
\end{equation}
Since $v$ satisfies the $m$-th order non-degeneracy condition, $\tilde k_i$ is well-defined and $v^{(\tilde k_i)}(\tilde y_i) \neq 0$.

Expanding $v(y)$ in a Taylor series near $y=\tilde y_i$, we obtain
\begin{align*}
v(y) &= v(\tilde y_i) + \frac{v^{(\tilde k_i)}(\tilde y_i)}{\tilde k_i!}(y-\tilde y_i)^{\tilde k_i}
       + o\!\left(|y-\tilde y_i|^{\tilde k_i}\right) \\
&= \lambda + \frac{v^{(\tilde k_i)}(\tilde y_i)}{\tilde k_i!}(y-\tilde y_i)^{\tilde k_i}
       + o\!\left(|y-\tilde y_i|^{\tilde k_i}\right).
\end{align*}
Hence, $\exists\,\varepsilon_i \in (0,1)$ such that, for any 
$y \in (\tilde y_i - \varepsilon_i,\tilde y_i + \varepsilon_i)\cap [L_1,L_2]$, we have
\begin{equation}\label{yi}
\frac{|v^{(\tilde k_i)}(\tilde y_i)|}{2\tilde k_i!}\,|y-\tilde y_i|^{\tilde k_i}
\;\leq\; |v(y)-\lambda|
\;\leq\; \frac{3|v^{(\tilde k_i)}(\tilde y_i)|}{2\tilde k_i!}\,|y-\tilde y_i|^{\tilde k_i}.
\end{equation}
Moreover, by the continuity of $v^{(\tilde k_i)}(y)$, $\exists 
\,\varepsilon_i' \in (0,1)$ such that, for any 
$y \in (\tilde y_i - \varepsilon_i',\tilde y_i + \varepsilon_i')\cap [L_1,L_2]$, we have
\begin{equation}\label{vki}
\bigl|v^{(\tilde k_i)}(y) - v^{(\tilde k_i)}(\tilde y_i)\bigr|
< \frac{|v^{(\tilde k_i)}(\tilde y_i)|}{2}.
\end{equation}

Define
\begin{align*}
\varepsilon_0 &= \min\bigl\{\,\varepsilon_i : 1 \leq i \leq N(\lambda)\bigr\}, \\
\varepsilon_0' &= \min\bigl\{\,\varepsilon_i' : 1 \leq i \leq N(\lambda)\bigr\}, \\
\varepsilon_0'' &= \tfrac{1}{4}\,\min\bigl\{\,|\tilde y_i - \tilde y_j| : 1 \leq i,j \leq N(\lambda),\ i \neq j \bigr\}.
\end{align*}

 Taking into account the relative positions of $\tilde y_1$ and $L_1$, as well as $\tilde y_{N(\lambda)}$ and $L_2$, we distinguish the following four cases:
\begin{align}\label{si}
\begin{split}
&\text{Case a.1:}\ \tilde y_1 = L_1,\ \tilde y_{N(\lambda)} = L_2;\\
&\text{Case a.2:}\ \tilde y_1 > L_1,\ \tilde y_{N(\lambda)} = L_2;\\
&\text{Case a.3:}\ \tilde y_1 = L_1,\ \tilde y_{N(\lambda)} < L_2;\\
&\text{Case a.4:}\ \tilde y_1 > L_1,\ \tilde y_{N(\lambda)} < L_2.
\end{split}
\end{align}
In the four different cases given in \eqref{si}, we define $\varepsilon(\lambda)$ as follows:
\begin{equation}
\varepsilon(\lambda) =
\begin{cases}
\dfrac{1}{2}\min\{\varepsilon_0,\varepsilon_0',\varepsilon_0''\}, 
& \tilde y_1 = L_1,\ \tilde y_{N(\lambda)} = L_2, \\[1ex]
\dfrac{1}{2}\min\{\varepsilon_0,\varepsilon_0',\varepsilon_0'',\ \tilde y_1 - L_1\}, 
& \tilde y_1 > L_1,\ \tilde y_{N(\lambda)} = L_2, \\[1ex]
\dfrac{1}{2}\min\{\varepsilon_0,\varepsilon_0',\varepsilon_0'',\ L_2 - \tilde y_{N(\lambda)}\}, 
& \tilde y_1 = L_1,\ \tilde y_{N(\lambda)} < L_2, \\[1ex]
\dfrac{1}{2}\min\{\varepsilon_0,\varepsilon_0',\varepsilon_0'',\ \tilde y_1 - L_1,\ L_2 - \tilde y_{N(\lambda)}\}, 
& \tilde y_1 > L_1,\ \tilde y_{N(\lambda)} < L_2.
\end{cases}
\end{equation}
Thus $\varepsilon_0,\varepsilon_0',\varepsilon_0''$, and $\varepsilon(\lambda)$ are all strictly positive.  
For brevity, we shall denote $\varepsilon(\lambda)$ simply by $\varepsilon$,  
with the understanding that its expression differs slightly depending on the four cases.

Therefore, for any $i,j \in \{1,\ldots,N(\lambda)\}$ with $i \neq j$, we have
\begin{equation}\label{bj}
[\tilde y_i - \varepsilon,\ \tilde y_i + \varepsilon] 
\;\cap\; 
[\tilde y_j - \varepsilon,\ \tilde y_j + \varepsilon] 
= \varnothing.
\end{equation}
Moreover, $\tilde y_1 - \varepsilon 
< \tilde y_1 + \varepsilon 
< \tilde y_2 - \varepsilon 
< \tilde y_2 + \varepsilon 
< \cdots 
< \tilde y_{N(\lambda)} - \varepsilon 
< \tilde y_{N(\lambda)} + \varepsilon.$

When $N(\lambda) \geq 2$, for any $i \in \{1,\ldots, N(\lambda)-1\}$,  
since $|v(y)-\lambda|$ is continuous with respect to $y$ on $[\tilde y_i + \varepsilon,\ \tilde y_{i+1} - \varepsilon]$,  
it attains a minimum value, which we denote by
\[
\eta_i = \min\bigl\{\,|v(y)-\lambda| : y \in [\tilde y_i + \varepsilon,\ \tilde y_{i+1} - \varepsilon] \bigr\}.
\]
We deduce that $\eta_i > 0$.  
Otherwise, by the definition of the minimum, there would exist some 
$\hat y \in [\tilde y_i + \varepsilon,\ \tilde y_{i+1} - \varepsilon]$ such that 
$|v(\hat y) - \lambda| = 0$, i.e. $v(\hat y) = \lambda$, which implies 
$\hat y \in E_\lambda = \{\tilde y_1,\tilde y_2,\ldots,\tilde y_{N(\lambda)}\}$. However, since
\[
\tilde y_1 < \tilde y_2 < \cdots < \tilde y_i 
< \tilde y_i + \varepsilon 
< \hat y 
< \tilde y_{i+1} - \varepsilon 
< \tilde y_{i+1} 
< \cdots < \tilde y_{N(\lambda)},
\]
it follows that $\hat y$ cannot coincide with any element of $E_\lambda$.  
This contradicts the fact that $\hat y \in E_\lambda$.  
Therefore, $\eta_i > 0$.

When $\tilde y_1 > L_1$, let $\eta_0 = \min\{\, |v(y) - \lambda| : y \in [L_1,\, \tilde y_1 - \varepsilon] \}$. By the same reasoning as above, we obtain $\eta_0 > 0$. When $\tilde y_{N(\lambda)} < L_2$, let $\eta_{N(\lambda)} = \min\{\, |v(y) - \lambda| : y \in [\tilde y_{N(\lambda)} + \varepsilon,\, L_2] \}$. Again, we obtain $\eta_{N(\lambda)} > 0$.

In the four different cases given in \eqref{si}, we define $\tilde{\eta}(\lambda)$ as follows:
\begin{equation}\label{sg}
\tilde{\eta}(\lambda) =
\begin{cases}
\dfrac{1}{4}\min\{\eta_1,\eta_2,\ldots,\eta_{N(\lambda)-1},\,1\}, 
& \tilde y_1 = L_1,\ \tilde y_{N(\lambda)} = L_2, \\[1ex]
\dfrac{1}{4}\min\{\eta_0,\eta_1,\eta_2,\ldots,\eta_{N(\lambda)-1},\,1\}, 
& \tilde y_1 > L_1,\ \tilde y_{N(\lambda)} = L_2, \\[1ex]
\dfrac{1}{4}\min\{\eta_1,\eta_2,\ldots,\eta_{N(\lambda)-1},\eta_{N(\lambda)},\,1\}, 
& \tilde y_1 = L_1,\ \tilde y_{N(\lambda)} < L_2, \\[1ex]
\dfrac{1}{4}\min\{\eta_0,\eta_1,\eta_2,\ldots,\eta_{N(\lambda)-1},\eta_{N(\lambda)},\,1\}, 
& \tilde y_1 > L_1,\ \tilde y_{N(\lambda)} < L_2.
\end{cases}
\end{equation}
Thus $\tilde{\eta}(\lambda) \in (0,1)$, and consequently
\begin{equation}\label{dl}
\tilde\delta_0(\la) := \tilde{\eta}(\lambda)^{1/m} \in (0,1).
\end{equation}

Take any
\begin{equation}\label{rq}
\lambda' \in (\lambda - (\tilde\delta_0(\lambda))^m,\ \lambda + (\tilde\delta_0(\lambda))^m) 
= (\lambda - \tilde{\eta}(\lambda),\ \lambda + \tilde{\eta}(\lambda)),
\end{equation}
and any $\delta \in (0,\tilde\delta_0(\lambda)]$.  
Consider the case a.4 in \eqref{si}, namely $\tilde y_1 > L_1$ and 
$\tilde y_{N(\lambda)} < L_2$. For any $y \in [L_1,\tilde y_1 - \varepsilon]$, we have
\begin{align*}
|v(y) - \lambda'| 
&\geq |v(y) - \lambda| - |\lambda - \lambda'| 
\;\geq\; \eta_0 - \tilde{\eta}(\lambda) 
\;>\; 2\tilde{\eta}(\lambda) 
= 2(\tilde\delta_0(\lambda))^m
>\; \delta^m.
\end{align*}
Hence $y \notin E_{\lambda',\delta}^m$.  
But since $y \in [L_1,L_2]$, it follows that $y \in [L_1,L_2]\setminus E_{\lambda',\delta}^m$.  
Therefore,
\[
[L_1,\tilde y_1] \;\subset\; [L_1,L_2]\setminus E_{\lambda',\delta}^m.
\]

Similarly,
\[
[\tilde y_{N(\lambda)} + \varepsilon,\ L_2] 
\;\subset\; [L_1,L_2]\setminus E_{\lambda',\delta}^m.
\]

When $N(\lambda) \geq 2$, similarly, for any 
$i \in \{1,\ldots, N(\lambda)-1\}$, we have
\[
[\tilde y_i + \varepsilon,\ \tilde y_{i+1} - \varepsilon] 
\;\subset\; [L_1,L_2]\setminus E_{\lambda',\delta}^m.
\]

Consequently,
\begin{align*}
&[L_1,L_2] \setminus \Bigl(\,\bigcup_{i=1}^{N(\lambda)}(\tilde y_i - \varepsilon,\ \tilde y_i + \varepsilon)\Bigr) \\
= &[L_1,\tilde y_1 - \varepsilon]
   \;\cup\; \Bigl(\,\bigcup_{i=1}^{N(\lambda)-1} [\tilde y_i + \varepsilon,\ \tilde y_{i+1} - \varepsilon] \Bigr)
   \;\cup\; [\tilde y_{N(\lambda)} + \varepsilon, L_2] \subset [L_1,L_2]\setminus E_{\lambda',\delta}^m,
\end{align*}
and therefore
\begin{equation}\label{ela}
E_{\lambda',\delta}^m \;\subset\; \bigcup_{i=1}^{N(\lambda)}(\tilde y_i - \varepsilon,\ \tilde y_i + \varepsilon).
\end{equation}
When $N(\lambda) = 1$, a similar argument shows that \eqref{ela} still holds.

Next, according to \eqref{ela}, we have
\begin{align}\label{jj}
\begin{split}
E_{\lambda',\delta}^m
&= \Bigl(\,\bigcup_{i=1}^{N(\lambda)}(\tilde y_i - \varepsilon,\ \tilde y_i + \varepsilon)\Bigr)\cap E_{\lambda',\delta}^m = \bigcup_{i=1}^{N(\lambda)}\Bigl((\tilde y_i - \varepsilon,\ \tilde y_i + \varepsilon)\cap E_{\lambda',\delta}^m\Bigr) =: \bigcup_{i=1}^{N(\lambda)} E^{m,i}_{\lambda',\delta}.
\end{split}
\end{align}
Since $E^{m,i}_{\lambda',\delta} \subset (\tilde y_i - \varepsilon,\ \tilde y_i + \varepsilon)$,  
by \eqref{bj} it follows that, for any $i,j \in \{1,\ldots,N(\lambda)\}$ with $i \neq j$, 
\[
E^{m,i}_{\lambda',\delta} \cap E^{m,j}_{\lambda',\delta} 
\subset 
(\tilde y_i - \varepsilon,\ \tilde y_i + \varepsilon) 
\cap (\tilde y_j - \varepsilon,\ \tilde y_j + \varepsilon) 
= \varnothing.
\]
Thus the sets $E^{m,i}_{\lambda',\delta}$ are pairwise disjoint.  
Combining this with \eqref{jj}, we obtain
\begin{equation}\label{qh}
m(E_{\lambda',\delta}^m)= \sum_{i=1}^{N(\lambda)} m(E^{m,i}_{\lambda',\delta}).
\end{equation}

Take any $i \in \{1,\ldots, N(\lambda)\}$.  
Without loss of generality, assume that $v^{(\tilde k_i)}(\tilde y_i) > 0$ (the case $<0$ is analogous).  
We distinguish between the cases where \underline{\textbf{$\tilde k_i$ is odd}} and where \underline{\textbf{$\tilde k_i$ is even}}.
\begin{itemize}
    \item[\tcr{\textbf{(A)}}] \textbf{If $\tilde k_i$ is odd, then $\tilde k_i - 1$ is even.}  
By \eqref{vki}, for any $y \in [\tilde y_i - \varepsilon,\ \tilde y_i + \varepsilon]$, we have
\[
\bigl| v^{(\tilde k_i)}(y) - v^{(\tilde k_i)}(\tilde y_i) \bigr|
< \tfrac{1}{2} v^{(\tilde k_i)}(\tilde y_i).
\]
It follows that
\begin{equation}\label{ss}
\tfrac{1}{2} v^{(\tilde k_i)}(\tilde y_i)
< v^{(\tilde k_i)}(y)
< \tfrac{3}{2} v^{(\tilde k_i)}(\tilde y_i).
\end{equation}

When $y \in (\tilde y_i,\ \tilde y_i + \varepsilon]$, integrating \eqref{ss} over $[\tilde y_i,y]$ gives
\[
\int_{\tilde y_i}^y \tfrac{1}{2} v^{(\tilde k_i)}(\tilde y_i)\,dz
\;\leq\; \int_{\tilde y_i}^y v^{(\tilde k_i)}(z)\,dz
\;\leq\; \int_{\tilde y_i}^y \tfrac{3}{2} v^{(\tilde k_i)}(\tilde y_i)\,dz.
\]
Since $v^{(\tilde k_i-1)}(\tilde y_i)=0$, it follows that
\begin{equation}\label{sss}
\tfrac{1}{2} v^{(\tilde k_i)}(\tilde y_i)(y-\tilde y_i)
\;\leq\; v^{(\tilde k_i-1)}(y)
\;\leq\; \tfrac{3}{2} v^{(\tilde k_i)}(\tilde y_i)(y-\tilde y_i).
\end{equation}
Integrating \eqref{sss} again over $[\tilde y_i,y]$, we obtain
\[
\int_{\tilde y_i}^y \tfrac{1}{2} v^{(\tilde k_i)}(\tilde y_i)(z-\tilde y_i)\,dz
\;\leq\; \int_{\tilde y_i}^y v^{(\tilde k_i-1)}(z)\,dz
\;\leq\; \int_{\tilde y_i}^y \tfrac{3}{2} v^{(\tilde k_i)}(\tilde y_i)(z-\tilde y_i)\,dz,
\]
that is, 
\[
\tfrac{1}{4} v^{(\tilde k_i)}(\tilde y_i)(y-\tilde y_i)^2
\;\leq\; v^{(\tilde k_i-2)}(y)
\;\leq\; \tfrac{3}{4} v^{(\tilde k_i)}(\tilde y_i)(y-\tilde y_i)^2.
\]
Proceeding in this way, we finally obtain
\[
\frac{v^{(\tilde k_i)}(\tilde y_i)}{2(\tilde k_i-1)!}(y-\tilde y_i)^{\tilde k_i-1}
\;\leq\; v'(y)
\;\leq\; \frac{3 v^{(\tilde k_i)}(\tilde y_i)}{2(\tilde k_i-1)!}(y-\tilde y_i)^{\tilde k_i-1}.
\]
Let
\begin{equation}\label{k}
\kappa_i = \frac{v^{(\tilde k_i)}(\tilde y_i)}{2(\tilde k_i-1)!}, 
\qquad 
\kappa'_i = \frac{3 v^{(\tilde k_i)}(\tilde y_i)}{2(\tilde k_i-1)!},
\end{equation}
so that
\begin{equation}\label{ka}
\kappa_i (y-\tilde y_i)^{\tilde k_i-1}
\;\leq\; v'(y)
\;\leq\; \kappa'_i (y-\tilde y_i)^{\tilde k_i-1}.
\end{equation}
Similarly, one can show that when $y \in [\tilde y_i - \varepsilon,\ \tilde y_i)$,  
\eqref{ka} still holds. Moreover, by the continuity of $v'$, the inequality  
\eqref{ka} also holds at $y = \tilde y_i$. Therefore,
\begin{equation}\label{yy}
\forall\, y \in [\tilde y_i - \varepsilon,\ \tilde y_i + \varepsilon],\quad
\kappa_i (y - \tilde y_i)^{\tilde k_i - 1}
\;\leq\; v'(y)
\;\leq\; \kappa'_i (y - \tilde y_i)^{\tilde k_i - 1}.
\end{equation}

Therefore, for any $y \in [\tilde y_i - \varepsilon,\ \tilde y_i) \cup (\tilde y_i,\ \tilde y_i + \varepsilon]$,  
we have $v'(y) > 0$. Since $v(y)$ is continuous at $y=\tilde y_i$,  
it follows that $v$ is strictly increasing on $[\tilde y_i - \varepsilon,\ \tilde y_i + \varepsilon]$.  
Thus,
\begin{align}\label{dx}
\begin{split}
v(\tilde y_i + \varepsilon) &= \max_{y \in [\tilde y_i - \varepsilon,\ \tilde y_i + \varepsilon]} v(y) > v(\tilde y_i) = \lambda, \\
v(\tilde y_i - \varepsilon) &= \min_{y \in [\tilde y_i - \varepsilon,\ \tilde y_i + \varepsilon]} v(y) < v(\tilde y_i) = \lambda.
\end{split}
\end{align}
Let $\eta_i' = \min\{\, \lambda - v(\tilde y_i - \varepsilon),\; v(\tilde y_i + \varepsilon) - \lambda \,\}$. Then, by \eqref{dx}, we have $\eta_i' > 0$.

Define
\[
\hat\delta_i(\lambda) = \min\Bigl\{\, \tilde{\delta}_0(\lambda),\; (\eta_i')^{1/m} \Bigr\}.
\]
By \eqref{dl}, then $\hat\delta_i(\lambda) \in (0,1)$. $\forall\lambda' \in \bigl(\lambda - (\hat\delta_i(\lambda))^m,\ \lambda + (\hat\delta_i(\lambda))^m\bigr) 
\subset \bigl(v(\tilde y_i - \varepsilon),\ v(\tilde y_i + \varepsilon)\bigr)$,  
by the intermediate value theorem for continuous functions, there exists 
$y_i(\lambda') \in (\tilde y_i - \varepsilon,\ \tilde y_i + \varepsilon)$ such that
\begin{equation}\label{wyi}
v\bigl(y_i(\lambda')\bigr) = \lambda'.
\end{equation}
Moreover, since $v$ is strictly increasing on $[\tilde y_i - \varepsilon,\ \tilde y_i + \varepsilon]$,  
the point $y_i(\lambda')$ satisfying \eqref{wyi} is unique in this interval.  
Again, because $v$ is strictly increasing on $[\tilde y_i - \varepsilon,\ \tilde y_i + \varepsilon]$, we have
\begin{equation}\label{xc}
(\lambda' - \lambda)\,\bigl(y_i(\lambda') - \tilde y_i\bigr) 
= \bigl(v(y_i(\lambda')) - v(\tilde y_i)\bigr)\,\bigl(y_i(\lambda') - \tilde y_i\bigr) \;\geq\; 0.
\end{equation}

Note that for all $\lambda' \in (\lambda - (\hat\delta_i(\lambda))^m,\ \lambda + (\hat\delta_i(\lambda))^m)$  
and all $\delta \in (0,\hat\delta_i(\lambda)]$, there are five possible cases:
\begin{align}\label{wu}
\begin{split}
&\text{Case b.1:}\quad \lambda + \delta^{\tilde k_i} \leq \lambda' < \lambda + (\hat\delta_i(\lambda))^m, \\[0.5ex]
&\text{Case b.2:}\quad \lambda < \lambda' < \lambda + \min\{\delta^{\tilde k_i},\, (\hat\delta_i(\lambda))^m\}, \\[0.5ex]
&\text{Case b.3:}\quad \lambda' = \lambda, \\[0.5ex]
&\text{Case b.4:}\quad \lambda - \min\{\delta^{\tilde k_i},\, (\hat\delta_i(\lambda))^m\} < \lambda' < \lambda, \\[0.5ex]
&\text{Case b.5:}\quad \lambda - (\hat\delta_i(\lambda))^m < \lambda' \leq \lambda - \delta^{\tilde k_i}.
\end{split}
\end{align}
In fact, when $\delta^{\tilde k_i} \geq (\hat\delta_i(\lambda))^m$,  
the values of $\lambda'$ corresponding to Cases b.1 and b.5 do not actually exist. When $\lambda < \lambda' < \lambda + (\hat\delta_i(\lambda))^m$,  
corresponding to Cases b.1 and b.2 in \eqref{wu}, it follows from \eqref{xc} that
\begin{equation}\label{zy}
\tilde y_i(\lambda') - \tilde y_i > 0, 
\quad \text{i.e. } \; y_i(\lambda') > \tilde y_i.
\end{equation}
\begin{itemize}
    \item[\textbf{(A.1)}] \textbf{Cases b.1 and b.5.} When $\lambda + \delta^{\tilde k_i} \leq \lambda' < \lambda + (\hat\delta_i(\lambda))^m$,  
that is, in Case b.1 of \eqref{wu},
\begin{equation}\label{EE}
\la'-\la\geq\de^{\tilde{k}_i}.
\end{equation}
Take any $y \in E^{m,i}_{\lambda',\delta}$, then $|v(y) - \lambda'| < \delta^m \leq \delta^{\tilde k_i}$, which implies
\[
0 \leq \lambda' - \lambda - \delta^{\tilde k_i}
< v(y) - \lambda
< \lambda' - \lambda + \delta^{\tilde k_i}.
\]
Hence,
\begin{equation}\label{dy}
v(y) > \lambda = v(\tilde y_i).
\end{equation}

Combining this with the fact that $v$ is strictly increasing on 
$[\tilde y_i - \varepsilon,\ \tilde y_i + \varepsilon]$, then
\begin{equation}\label{dddz}
y > \tilde y_i.
\end{equation}
Note that for any $y \in [\tilde y_i - \varepsilon,\ \tilde y_i]$,  
using the monotonicity of $v$ on $[\tilde y_i - \varepsilon,\ \tilde y_i + \varepsilon]$, we have $v(y) \leq v(\tilde y_i) = \lambda$, that is,
\begin{equation}\label{xy}
\forall\, y \in [\tilde y_i - \varepsilon,\ \tilde y_i], 
\quad v(y) \leq \lambda = v(\tilde y_i).
\end{equation}

Combining \eqref{dy} with \eqref{xy}, when 
$\lambda + \delta^{\tilde k_i} < \lambda' < \lambda + (\hat\delta_i(\lambda))^m$,  
for any $y \in E^{m,i}_{\lambda',\delta}$ we must have $y \notin [\tilde y_i - \varepsilon,\ \tilde y_i]$.  
Together with the definition of $E^{m,i}_{\lambda',\delta}$ in \eqref{jj}, this implies $y \in E^{m,i}_{\lambda',\delta} \setminus [\tilde y_i - \varepsilon,\ \tilde y_i] 
\;\subset\; (\tilde y_i - \varepsilon,\ \tilde y_i + \varepsilon) \setminus [\tilde y_i - \varepsilon,\ \tilde y_i] 
= (\tilde y_i,\ \tilde y_i + \varepsilon)$. Therefore,
\[
E^{m,i}_{\lambda',\delta} \;\subset\; (\tilde y_i,\ \tilde y_i + \varepsilon).
\]
\quad Since \eqref{zy} shows that $\tilde y_i < y_i(\lambda')$, integrating \eqref{ka} over 
$[\tilde y_i,\ y_i(\lambda')]$ gives
\[
\int_{\tilde y_i}^{y_i(\lambda')} \kappa_i (y-\tilde y_i)^{\tilde k_i - 1}\,dy
\;\leq\; \int_{\tilde y_i}^{y_i(\lambda')} v'(y)\,dy
\;\leq\; \int_{\tilde y_i}^{y_i(\lambda')} \kappa'_i (y-\tilde y_i)^{\tilde k_i - 1}\,dy,
\]
which yields
\[
\frac{\kappa_i}{\tilde k_i}\,(y_i(\lambda') - \tilde y_i)^{\tilde k_i}
\;\leq\; v(y_i(\lambda')) - v(\tilde y_i)
\;\leq\; \frac{\kappa'_i}{\tilde k_i}\,(y_i(\lambda') - \tilde y_i)^{\tilde k_i}.
\]

Since $v(y_i(\lambda')) = \lambda'$ and $v(\tilde y_i) = \lambda$, we obtain
\begin{equation}\label{lb}
\left(\frac{\tilde k_i(\lambda' - \lambda)}{\kappa'_i}\right)^{1/\tilde k_i}
\;\leq\; y_i(\lambda') - \tilde y_i
\;\leq\; \left(\frac{\tilde k_i(\lambda' - \lambda)}{\kappa_i}\right)^{1/\tilde k_i}.
\end{equation}

\quad For any $y$ satisfying $y_i(\lambda') \leq y \leq \tilde y_i + \varepsilon$,  
integrating \eqref{ka} over $[y_i(\lambda'),\, y]$ gives
\[
\int_{y_i(\lambda')}^{y} \kappa_i (z-\tilde y_i)^{\tilde k_i - 1}\,dz
\;\leq\; \int_{y_i(\lambda')}^{y} v'(z)\,dz
\;\leq\; \int_{y_i(\lambda')}^{y} \kappa'_i (z-\tilde y_i)^{\tilde k_i - 1}\,dz,
\]
hence
\begin{equation}\label{zuo}
\frac{\kappa_i}{\tilde k_i}\Bigl[(y-\tilde y_i)^{\tilde k_i} - (y_i(\lambda')-\tilde y_i)^{\tilde k_i}\Bigr]
\;\leq\; v(y) - \lambda'
\;\leq\; \frac{\kappa'_i}{\tilde k_i}\Bigl[(y-\tilde y_i)^{\tilde k_i} - (y_i(\lambda')-\tilde y_i)^{\tilde k_i}\Bigr].
\end{equation}
Similarly, for any $y$ satisfying $\tilde y_i - \varepsilon \leq y \leq y_i(\lambda')$, we have
\begin{equation}\label{you}
\frac{\kappa_i}{\tilde k_i}\Bigl[(y-\tilde y_i)^{\tilde k_i} - (y_i(\lambda')-\tilde y_i)^{\tilde k_i}\Bigr]
\;\geq\; v(y) - \lambda'
\;\geq\; \frac{\kappa'_i}{\tilde k_i}\Bigl[(y-\tilde y_i)^{\tilde k_i} - (y_i(\lambda')-\tilde y_i)^{\tilde k_i}\Bigr].
\end{equation}
\quad Combining \eqref{zuo} and \eqref{you}, we deduce that for any 
$y \in [\tilde y_i - \varepsilon,\ \tilde y_i + \varepsilon]$,  
\begin{equation}\label{jdz}
\frac{\kappa_i}{\tilde k_i}
\left| (y - \tilde y_i)^{\tilde k_i} - (y_i(\lambda') - \tilde y_i)^{\tilde k_i} \right|
\;\leq\; |v(y) - \lambda'|
\;\leq\; \frac{\kappa'_i}{\tilde k_i}
\left| (y - \tilde y_i)^{\tilde k_i} - (y_i(\lambda') - \tilde y_i)^{\tilde k_i} \right|.
\end{equation}
Note that $\delta \in (0,\hat\delta_i(\lambda)] \subset (0,1)$.  
Thus, when $y \in E^{m,i}_{\lambda',\delta}$ we have $|v(y) - \lambda'| < \delta^m$.  
Using \eqref{jdz}, it follows that
\begin{equation}\label{zj}
\frac{\kappa_i}{\tilde k_i}
\left| (y - \tilde y_i)^{\tilde k_i} - (y_i(\lambda') - \tilde y_i)^{\tilde k_i} \right|
\;\leq\; |v(y) - \lambda'|
\;<\; \delta^m \;\leq\; \delta^{\tilde k_i}.
\end{equation}

From \eqref{zy} we know that $y_i(\lambda') - \tilde y_i > 0$.  Take any $y \in E^{m,i}_{\lambda',\delta}$.  
By \eqref{dddz} we also have $y - \tilde y_i > 0$. Then
\begin{align*}
&|y - y_i(\lambda')|
= \bigl|(y - \tilde y_i) - (y_i(\lambda') - \tilde y_i)\bigr| \\
=& \frac{\bigl|(y - \tilde y_i)^{\tilde k_i} - (y_i(\lambda') - \tilde y_i)^{\tilde k_i}\bigr|}
{(y - \tilde y_i)^{\tilde k_i - 1} + (y - \tilde y_i)^{\tilde k_i - 2}(y_i(\lambda') - \tilde y_i) + \cdots + (y_i(\lambda') - \tilde y_i)^{\tilde k_i - 1}} \\
\leq& \frac{\bigl|(y - \tilde y_i)^{\tilde k_i} - (y_i(\lambda') - \tilde y_i)^{\tilde k_i}\bigr|}
{(y_i(\lambda') - \tilde y_i)^{\tilde k_i - 1}} \overset{\eqref{zj}}{<} \frac{\tilde k_i\,\delta^{\tilde k_i}}
{\kappa_i (y_i(\lambda') - \tilde y_i)^{\tilde k_i - 1}} 
\overset{\eqref{lb}}{\leq} \frac{\tilde k_i\,\delta^{\tilde k_i}}
{\kappa_i \left(\frac{\tilde k_i(\lambda' - \lambda)}{\kappa'_i}\right)^{(\tilde k_i - 1)/\tilde k_i}} \\
\overset{\eqref{EE}}{\leq}& \frac{\tilde k_i}{\kappa_i}
\cdot \frac{\delta^{\tilde k_i}}
{\left(\tfrac{\tilde k_i \delta^{\tilde k_i}}{\kappa'_i}\right)^{1 - 1/\tilde k_i}} = \frac{\tilde k_i}{\kappa_i}
\left(\frac{\kappa'_i}{\tilde k_i}\right)^{1 - 1/\tilde k_i} \delta,
\end{align*}
that is,
\begin{equation}\label{cjc}
|y - y_i(\lambda')|
< \frac{\tilde k_i}{\kappa_i}
\left(\frac{\kappa'_i}{\tilde k_i}\right)^{1 - 1/\tilde k_i} \,\delta.
\end{equation}
Therefore,
\begin{align}\label{jia}
\begin{split}
E^{m,i}_{\lambda',\delta}
&\subset \left(
y_i(\lambda') - \frac{\tilde k_i}{\kappa_i}
\left(\frac{\kappa'_i}{\tilde k_i}\right)^{1 - 1/\tilde k_i}\delta,\,
y_i(\lambda') + \frac{\tilde k_i}{\kappa_i}
\left(\frac{\kappa'_i}{\tilde k_i}\right)^{1 - 1/\tilde k_i}\delta
\right).
\end{split}
\end{align}
Hence,
\begin{equation}\label{cedu}
m(E^{m,i}_{\lambda',\delta})
\;\leq\; 2 \cdot \frac{\tilde k_i}{\kappa_i}
\left(\frac{\kappa'_i}{\tilde k_i}\right)^{1 - 1/\tilde k_i}\delta.
\end{equation}

\quad The above analysis was carried out for Case b.1 in \eqref{wu},  
but by a similar argument, both \eqref{jia} and \eqref{cedu} remain valid  
in Case b.5 as well.
\item[\textbf{(A.2)}] \textbf{Cases b.2 and b.4.} For Case b.2 in \eqref{wu}, namely when $0 < \lambda' - \lambda \leq \delta^{\tilde k_i}$,  
for any $y \in E^{m,i}_{\lambda',\delta}$ we have
\begin{align*}
|y - \tilde y_i|^{\tilde k_i}
&\leq \bigl||y - \tilde y_i|^{\tilde k_i} - |y_i(\lambda') - \tilde y_i|^{\tilde k_i}\bigr|
    + |y_i(\lambda') - \tilde y_i|^{\tilde k_i} \\
&\leq \bigl|(y - \tilde y_i)^{\tilde k_i} - (y_i(\lambda') - \tilde y_i)^{\tilde k_i}\bigr|
    + |y_i(\lambda') - \tilde y_i|^{\tilde k_i} \\
&\overset{\eqref{jdz}}{\leq} \frac{\tilde k_i}{\kappa_i}|v(y) - \lambda'|
    + |y_i(\lambda') - \tilde y_i|^{\tilde k_i} \overset{\eqref{zj}}{<} \frac{\tilde k_i}{\kappa_i}\,\delta^{\tilde k_i}
    + |y_i(\lambda') - \tilde y_i|^{\tilde k_i} \\
&\overset{\eqref{lb}}{\leq} \frac{\tilde k_i}{\kappa_i}\,\delta^{\tilde k_i}
    + \frac{\tilde k_i(\lambda' - \lambda)}{\kappa_i} \overset{\eqref{rq}}{<} \frac{\tilde k_i}{\kappa_i}\,\delta^{\tilde k_i}
    + \frac{\tilde k_i}{\kappa_i}\,\delta^{\tilde k_i} = \frac{2\tilde k_i}{\kappa_i}\,\delta^{\tilde k_i}.
\end{align*}
Hence,
\[
|y - \tilde y_i| \leq \left(\frac{2\tilde k_i}{\kappa_i}\right)^{1/\tilde k_i}\delta.
\]

Therefore, in Case b.2 of \eqref{wu}, we obtain
\begin{equation}\label{bc}
E^{m,i}_{\lambda',\delta} \;\subset\;
\left(\tilde y_i - \left(\tfrac{2\tilde k_i}{\kappa_i}\right)^{1/\tilde k_i}\delta,\;
      \tilde y_i + \left(\tfrac{2\tilde k_i}{\kappa_i}\right)^{1/\tilde k_i}\delta\right),
\end{equation}
and consequently
\begin{equation}\label{de}
m(E^{m,i}_{\lambda',\delta})
\;\leq\; 2\left(\tfrac{2\tilde k_i}{\kappa_i}\right)^{1/\tilde k_i}\delta.
\end{equation}

\quad By a similar argument, the estimates \eqref{bc}–\eqref{de} also hold in Case b.4.

\item[\textbf{(A.3)}] \textbf{Case b.3.} In particular, for Case b.3 in \eqref{wu}, namely when $\lambda' = \lambda$,  
we have $\tilde y_i = y_i(\lambda')$.  
For any $y \in E^{m,i}_{\lambda',\delta} \subset [\tilde y_i - \varepsilon,\ \tilde y_i + \varepsilon]$, it follows that
\[
\frac{\kappa_i}{\tilde k_i}\,|y - \tilde y_i|^{\tilde k_i}
\;\overset{\eqref{jdz}}{\leq}\; |v(y) - \lambda|
\;<\; \delta^m \;\leq\; \delta^{\tilde k_i},
\]
hence
\begin{equation}\label{bhy}
y \in \left(\tilde y_i - \left(\tfrac{\tilde k_i}{\kappa_i}\right)^{1/\tilde k_i}\delta,\;
           \tilde y_i + \left(\tfrac{\tilde k_i}{\kappa_i}\right)^{1/\tilde k_i}\delta\right).
\end{equation}
Therefore,
\begin{equation}\label{bhyy}
m(E^{m,i}_{\lambda',\delta})
\;\leq\; 2\left(\tfrac{\tilde k_i}{\kappa_i}\right)^{1/\tilde k_i}\delta
\;\leq\; 2\left(\tfrac{2\tilde k_i}{\kappa_i}\right)^{1/\tilde k_i}\delta.
\end{equation}
\end{itemize}

Combining all the cases discussed above, we define
\begin{align}\label{cde}
\begin{split}
C_i(\lambda) 
:= \max\left\{
2\frac{\tilde k_i}{\kappa_i}\left(\frac{\kappa'_i}{\tilde k_i}\right)^{1 - 1/\tilde k_i},\,
2\left(\frac{2\tilde k_i}{\kappa_i}\right)^{1/\tilde k_i}
\right\}.
\end{split}
\end{align}
By \eqref{dl}, we have $\hat\delta_i(\lambda) \in (0,1)$.  Combining \eqref{jia}, \eqref{cedu}, \eqref{bc}, \eqref{de}, \eqref{bhy}, and \eqref{bhyy},  
we conclude that for any $\lambda' \in (\lambda - (\hat\delta_i(\lambda))^m,\ 
\lambda + (\hat\delta_i(\lambda))^m)$ and any $\delta \in (0,\hat\delta_i(\lambda)]$,  
there exists an interval $I_i(\lambda,\delta)$ of length at most $C_i(\lambda)\delta$ such that
\begin{equation}\label{cf}
E^{m,i}_{\lambda',\delta} \subset I_i(\lambda,\delta),
\qquad
m(E^{m,i}_{\lambda',\delta}) \leq m(I_i(\lambda,\delta)) \leq C_i(\lambda)\delta.
\end{equation}

\textit{The above proof was carried out under the assumptions of Case a.4 in \eqref{si},  
but a similar argument shows that \eqref{cf} also holds in the other three cases.  Moreover, the analysis above was derived under the assumption that $v^{(\tilde k_i)}(\tilde y_i) > 0$,  
but by a completely analogous argument one can show that \eqref{cf} remains valid  
when $v^{(\tilde k_i)}(\tilde y_i) < 0$.}

\item[\textbf{(B)}] \textbf{ If $\tilde k_i$ is even, then $\tilde k_i - 1$ is odd.}  By the same integration argument as in the proof of \eqref{yy} for the case where $\tilde k_i$ is odd,  
we obtain
\begin{align}\label{yyd}
\begin{split}
\forall\, y \in [\tilde y_i,\ \tilde y_i + \varepsilon], 
&\quad \kappa_i (y - \tilde y_i)^{\tilde k_i - 1} 
   \;\leq\; v'(y) \;\leq\; \kappa'_i (y - \tilde y_i)^{\tilde k_i - 1}, \\[0.5ex]
\forall\, y \in [\tilde y_i - \varepsilon,\ \tilde y_i], 
&\quad \kappa_i (y - \tilde y_i)^{\tilde k_i - 1} 
   \;\geq\; v'(y) \;\geq\; \kappa'_i (y - \tilde y_i)^{\tilde k_i - 1}.
\end{split}
\end{align}
Here, the definitions of $\kappa_i$ and $\kappa'_i$ are the same as in \eqref{k}.

It follows that on the interval $[\tilde y_1' - \varepsilon, \tilde y_1')$ we have $v'(y) < 0$, while on $(\tilde y_1', \tilde y_1' + \varepsilon]$ we have $v'(y) > 0$. Since $v$ is continuous at $y = \tilde y_i$, it is strictly decreasing on $[\tilde y_1 - \varepsilon, \tilde y_1]$ and strictly increasing on $[\tilde y_1, \tilde y_1 + \varepsilon]$. Therefore, on the interval $[\tilde y_i - \varepsilon, \tilde y_i + \varepsilon]$ the function $v$ attains its minimum at $y = \tilde y_i$ and its maximum at either $y = \tilde y_i - \varepsilon$ or $y = \tilde y_i + \varepsilon$.

Based on the above analysis of the monotonicity of $v$, we obtain
\begin{align}\label{xd}
	\begin{split}
	v(\tilde y_i - \varepsilon) > v(\tilde y_i) = \lambda, \ \text{hence } v(\tilde y_i - \varepsilon) - \lambda > 0; \\
	v(\tilde y_i + \varepsilon) > v(\tilde y_i) = \lambda, \ \text{hence } v(\tilde y_i + \varepsilon) - \lambda > 0.
	\end{split}
\end{align}
Let
\[
\eta_i'' = \min\left\{\, v(\tilde y_i - \varepsilon) - \lambda, \; v(\tilde y_i + \varepsilon) - \lambda \,\right\}.
\]
By \eqref{xd}, we have $\eta_i'' > 0$. Define
\[
\hat{\delta}_i(\lambda) = \min\left\{\, \tilde\delta_0(\lambda), \; \bigl(\eta_i''\bigr)^{1/m} \,\right\},
\]
then by \eqref{dl}, $\hat{\delta}_i(\lambda) \in (0,1)$.

From the definition of $E^{m,i}_{\lambda',\delta}$ in \eqref{jj}, we have
\begin{align}\label{fg}
	\begin{split}
	E^{m,i}_{\lambda',\delta}
	&= E^{m,i}_{\lambda',\delta} \cap (\tilde y_i - \varepsilon, \tilde y_i + \varepsilon) \\
	&= \bigl(E^{m,i}_{\lambda',\delta} \cap (\tilde y_i - \varepsilon, \tilde y_i] \bigr) 
	    \cup \bigl(E^{m,i}_{\lambda',\delta} \cap (\tilde y_i, \tilde y_i + \varepsilon)\bigr) =: E^{m,i,-}_{\lambda',\delta} \cup E^{m,i,+}_{\lambda',\delta}.
	\end{split}
\end{align}
Noting that $E^{m,i,-}_{\lambda',\delta} \cap E^{m,i,+}_{\lambda',\delta} = \varnothing$, it follows that
\begin{equation}\label{jfa}
	m(E^{m,i}_{\lambda',\delta}) 
	= m(E^{m,i,-}_{\lambda',\delta}) + m(E^{m,i,+}_{\lambda',\delta}).
\end{equation}

For any $\lambda' \in \bigl(\lambda - (\hat{\delta}_i(\lambda))^m, \, \lambda + (\hat{\delta}_i(\lambda))^m \bigr)$ and any $\delta \in (0, \hat{\delta}_i(\lambda)]$,  
there are the following three cases:
\begin{align}\label{san}
	\begin{split}
&\text{Case c.1:}\quad \lambda + \delta^{\tilde k_i} \leq \lambda' < \lambda + (\hat{\delta}_i(\lambda))^m;\\
&\text{Case c.2:}\quad \lambda - \min\{\delta^{\tilde k_i}, (\hat{\delta}_i(\lambda))^m\} < \lambda' 
                   < \lambda + \min\{\delta^{\tilde k_i}, (\hat{\delta}_i(\lambda))^m\};\\
&\text{Case c.3:}\quad \lambda - (\hat{\delta}_i(\lambda))^m < \lambda' \leq \lambda + \delta^{\tilde k_i}.
	\end{split}
\end{align}
It is worth mentioning that when $\delta^{\tilde k_i} \geq (\hat{\delta}_i(\lambda))^m$, 
the values of $\lambda'$ satisfying Case~c.1 and Case~c.3 do not actually exist.

\begin{itemize}
    \item[\textbf{(B.1)}]\textbf{Case~c.1.} For Case~c.1, namely when $\lambda' \geq \lambda + \delta^{\tilde k_i} > \lambda$, so that
    \begin{equation}\label{yy}
    \la'-\la\geq\de^{\tilde{k}_i}.
    \end{equation}   
The monotonicity analysis of $v$ implies that the equation $v(y) = \lambda'$ 
has exactly two solutions in the interval $[\tilde y_i - \varepsilon, \tilde y_i + \varepsilon]$, 
denoted by $y_i^-(\lambda')$ and $y_i^+(\lambda')$, which satisfy
\begin{equation}\label{dxgx}
\tilde y_i - \varepsilon < y_i^-(\lambda') < \tilde y_i < y_i^+(\lambda') < \tilde y_i + \varepsilon.
\end{equation}

Same as the proof of \eqref{lb}and\eqref{zj}, we can also get\begin{align}\label{lbb}
\begin{split}
\left(\frac{\tilde k_i(\lambda' - \lambda)}{\kappa'_i}\right)^{1/\tilde k_i}
\;\leq\; y_i^+(\lambda') - \tilde y_i
\;\leq\; \left(\frac{\tilde k_i(\lambda' - \lambda)}{\kappa_i}\right)^{1/\tilde k_i},\\\left(\frac{\tilde k_i(\lambda' - \lambda)}{\kappa'_i}\right)^{1/\tilde k_i}
\;\leq\; \tilde y_i-y_i^-(\lambda') 
\;\leq\; \left(\frac{\tilde k_i(\lambda' - \lambda)}{\kappa_i}\right)^{1/\tilde k_i},
\end{split}
\end{align}
and
\begin{align}\label{lg}
\begin{split}
\forall\, y\in E^{m,i,+}_{\la',\de},\ \frac{\kappa_i}{\tilde k_i}
\left| (y - \tilde y_i)^{\tilde k_i} - (y_i^+(\lambda') - \tilde y_i)^{\tilde k_i} \right|
\;\leq\; |v(y) - \lambda'|
\;<\; \delta^m \;\leq\; \delta^{\tilde k_i}.\\
\forall\, y\in E^{m,i,-}_{\la',\de},\ \frac{\kappa_i}{\tilde k_i}
\left| (y - \tilde y_i)^{\tilde k_i} - (y_i^-(\lambda') - \tilde y_i)^{\tilde k_i} \right|
\;\leq\; |v(y) - \lambda'|
\;<\; \delta^m \;\leq\; \delta^{\tilde k_i}.
\end{split}
\end{align}

Then, for any $y \in E^{m,i,+}_{\lambda',\delta}$, we have $y > \tilde y_i$ and $y \in E^{m,i}_{\lambda',\delta}$. Then
\begin{align*}
		&|y - y_i^+(\lambda')| \\
		=& |(y - \tilde y_i) - (y_i^+(\lambda') - \tilde y_i)| \\
		=& \frac{|(y - \tilde y_i)^{\tilde k_i} - (y_i^+(\lambda') - \tilde y_i)^{\tilde k_i}|}
		       {(y - \tilde y_i)^{\tilde k_i - 1} + (y - \tilde y_i)^{\tilde k_i - 2}(y_i^+(\lambda') - \tilde y_i) + \ldots + (y_i^+(\lambda') - \tilde y_i)^{\tilde k_i - 1}} \\
		\leq& \frac{|(y - \tilde y_i)^{\tilde k_i} - (y_i^+(\lambda') - \tilde y_i)^{\tilde k_i}|}
		       {(y_i^+(\lambda') - \tilde y_i)^{\tilde k_i - 1}} 
		\overset{(\ref{lg})}{<} \frac{\tilde k_i \delta^{\tilde k_i}}
		       {\kappa_i (y_i^+(\lambda') - \tilde y_i)^{\tilde k_i - 1}} \\
		\overset{(\ref{lbb})}{\leq}& \frac{\tilde k_i \delta^{\tilde k_i}}
		       {\kappa_i \bigl(\tfrac{\tilde k_i (\lambda' - \lambda)}{\kappa_i'}\bigr)^{(\tilde k_i - 1)/\tilde k_i}} 
		\overset{(\ref{yy})}{\leq} \frac{\tilde k_i}{\kappa_i} \cdot
		       \frac{\delta^{\tilde k_i}}{\bigl(\tfrac{\tilde k_i \delta^{\tilde k_i}}{\kappa_i'}\bigr)^{1 - 1/\tilde k_i}} 
		= \frac{\tilde k_i}{\kappa_i} \Bigl(\frac{\kappa_i'}{\tilde k_i}\Bigr)^{1 - 1/\tilde k_i} \delta.
\end{align*}

Therefore,
\begin{equation}\label{xg}
	E^{m,i,+}_{\lambda',\delta} \subset 
	\Bigl(y_i^+(\lambda') - \frac{\tilde k_i}{\kappa_i}\Bigl(\frac{\kappa_i'}{\tilde k_i}\Bigr)^{1 - 1/\tilde k_i} \delta, \;
	      y_i^+(\lambda') + \frac{\tilde k_i}{\kappa_i}\Bigl(\frac{\kappa_i'}{\tilde k_i}\Bigr)^{1 - 1/\tilde k_i} \delta \Bigr),
\end{equation}
and hence
\begin{equation}\label{jyb}
	m(E^{m,i,+}_{\lambda',\delta}) 
	\leq 2 \frac{\tilde k_i}{\kappa_i}\Bigl(\frac{\kappa_i'}{\tilde k_i}\Bigr)^{1 - 1/\tilde k_i} \delta.
\end{equation}

\quad Similarly,
\begin{equation}\label{xgg}
	E^{m,i,-}_{\lambda',\delta} \subset 
	\Bigl(y_i^-(\lambda') - \frac{\tilde k_i}{\kappa_i}\Bigl(\frac{\kappa_i'}{\tilde k_i}\Bigr)^{1 - 1/\tilde k_i}\delta,\;
	      y_i^-(\lambda') + \frac{\tilde k_i}{\kappa_i}\Bigl(\frac{\kappa_i'}{\tilde k_i}\Bigr)^{1 - 1/\tilde k_i}\delta\Bigr),
\end{equation}
and hence
\begin{equation}\label{jybb}
	m(E^{m,i,-}_{\lambda',\delta}) 
	\leq 2 \frac{\tilde k_i}{\kappa_i}\Bigl(\frac{\kappa_i'}{\tilde k_i}\Bigr)^{1 - 1/\tilde k_i}\delta.
\end{equation}
Combining \eqref{jfa}, \eqref{jyb}, and \eqref{jybb}, we obtain
\begin{equation}\label{yyc}
	m(E^{m,i}_{\lambda',\delta})
	= m(E^{m,i,-}_{\lambda',\delta}) + m(E^{m,i,+}_{\lambda',\delta})
	\leq 4 \frac{\tilde k_i}{\kappa_i}\Bigl(\frac{\kappa_i'}{\tilde k_i}\Bigr)^{1 - 1/\tilde k_i}\delta.
\end{equation}
\item[\textbf{(B.2)}]\textbf{Case~c.2.} For Case~c.2, we have $\lambda' \in (\lambda - \delta^{\tilde k_i}, \lambda + \delta^{\tilde k_i})$.  
Then the equation in $y$,
\[
v(y) = \lambda' + \min\{\delta^{\tilde k_i}, (\hat{\delta}_i(\lambda))^m\},
\]
has exactly two solutions in $[\tilde y_i - \varepsilon, \tilde y_i + \varepsilon]$, denoted by $\hat y_i^-(\lambda')$ and $\hat y_i^+(\lambda')$, which satisfy
\begin{equation}\label{syu}
L_1 < \tilde y_i - \varepsilon \leq \hat y_i^-(\lambda') < \tilde y_i < \hat y_i^+(\lambda') \leq \tilde y_i + \varepsilon < L_2.
\end{equation}

By the monotonicity of $v$, we deduce that
\begin{equation}\label{zjd}
	E^{m,i,-}_{\lambda',\delta} = (\hat y_i^-(\lambda'), \tilde y_i], 
	\quad 
	E^{m,i,+}_{\lambda',\delta} = (\tilde y_i, \hat y_i^+(\lambda')).
\end{equation}

\quad For any $y \in E^{m,i,+}_{\lambda',\delta} \subset (\tilde y_i - \varepsilon_i, \tilde y_i + \varepsilon_i)\cap[L_1,L_2]$,  
\eqref{vki} holds, namely
\begin{equation*}
	|v^{(\tilde k_i)}(y) - v^{(\tilde k_i)}(\tilde y_i)| 
	< \frac{|v^{(\tilde k_i)}(\tilde y_i)|}{2} 
	= \frac{v^{(\tilde k_i)}(\tilde y_i)}{2},
\end{equation*}
which implies
\begin{equation}\label{jf}
	\frac{v^{(\tilde k_i)}(\tilde y_i)}{2} 
	< v^{(\tilde k_i)}(y) 
	< \frac{3 v^{(\tilde k_i)}(\tilde y_i)}{2}.
\end{equation}

Integrating \eqref{jf} over $(\tilde y_i,y]$, we obtain
\[
\int_{\tilde y_i}^y \frac{v^{(\tilde k_i)}(\tilde y_i)}{2}\,dz
\;\leq\; \int_{\tilde y_i}^y v^{(\tilde k_i)}(z)\,dz
\;\leq\; \int_{\tilde y_i}^y \frac{3 v^{(\tilde k_i)}(\tilde y_i)}{2}\,dz,
\]
and since $v^{(\tilde k_i - 1)}(\tilde y_i) = 0$, it follows that
\begin{equation}\label{jff}
	\frac{v^{(\tilde k_i)}(\tilde y_i)}{2}(y - \tilde y_i)
	\;\leq\; v^{(\tilde k_i - 1)}(y)
	\;\leq\; \frac{3 v^{(\tilde k_i)}(\tilde y_i)}{2}(y - \tilde y_i).
\end{equation}

Integrating \eqref{jff} over $(\tilde y_i,y]$, we obtain
\[
\int_{\tilde y_i}^y \frac{v^{(\tilde k_i)}(\tilde y_i)}{2}(z - \tilde y_i)\,dz
\;\leq\; \int_{\tilde y_i}^y v^{(\tilde k_i - 1)}(z)\,dz
\;\leq\; \int_{\tilde y_i}^y \frac{3 v^{(\tilde k_i)}(\tilde y_i)}{2}(z - \tilde y_i)\,dz.
\]

Using again that $v^{(\tilde k_i - 2)}(\tilde y_i) = 0$, we deduce
\[
\frac{v^{(\tilde k_i)}(\tilde y_i)}{4}(y - \tilde y_i)^2
\;\leq\; v^{(\tilde k_i - 2)}(y)
\;\leq\; \frac{3 v^{(\tilde k_i)}(\tilde y_i)}{4}(y - \tilde y_i)^2.
\]
Proceeding in this way, we finally obtain
\[
\frac{v^{(\tilde k_i)}(\tilde y_i)}{2(\tilde k_i - 1)!}(y - \tilde y_i)^{\tilde k_i - 1}
\;\leq\; v'(y) \;\leq\;
\frac{3 v^{(\tilde k_i)}(\tilde y_i)}{2(\tilde k_i - 1)!}(y - \tilde y_i)^{\tilde k_i - 1},
\]
which yields
\begin{equation}\label{kaa}
	\kappa_i (y - \tilde y_i)^{\tilde k_i - 1}
	\;\leq\; v'(y) 
	\;\leq\; \kappa_i'(y - \tilde y_i)^{\tilde k_i - 1}.
\end{equation}
Here, $\kappa_i$ and $\kappa_i'$ are the same as in the odd case, namely those given in \eqref{k}.

\quad Integrating \eqref{kaa} over $(\tilde y_i, y]$, we obtain
\begin{equation}\label{jfff}
	\int_{\tilde y_i}^{y} \kappa_i (z - \tilde y_i)^{\tilde k_i - 1}\,dz
	\;\leq\; \int_{\tilde y_i}^{y} v'(z)\,dz
	\;\leq\; \int_{\tilde y_i}^{y} \kappa_i' (z - \tilde y_i)^{\tilde k_i - 1}\,dz.
\end{equation}
Noting that $v(\tilde y_i) = \lambda$, it follows from \eqref{jfff} that
\begin{equation}\label{gj}
	0 \;\leq\; \frac{\kappa_i}{\tilde k_i}(\hat y_i^+(\lambda')- \tilde y_i)^{\tilde k_i}
	\;\leq\; v(y) - \lambda
	\;\leq\; \frac{\kappa_i'}{\tilde k_i}(\hat y_i^+(\lambda')- \tilde y_i)^{\tilde k_i}.
\end{equation}
Moreover, since $y \in E^{m,i,+}_{\lambda',\delta} \subset E^{m,i}_{\lambda',\delta}$, then $0 < v(y) - \lambda < \delta^m \leq \delta^{\tilde k_i}$. Combining this with \eqref{gj}, we obtain
\begin{equation}\label{zzz}
	\hat y_i^+(\lambda')- \tilde y_i 
	\;\leq\; \Bigl(\tfrac{\tilde k_i}{\kappa_i}\,(v(y)-\lambda)\Bigr)^{1/\tilde k_i}
	\;<\; \Bigl(\tfrac{\tilde k_i}{\kappa_i}\Bigr)^{1/\tilde k_i}\delta.
\end{equation}
According to \eqref{zjd}, 
\begin{equation}\label{xj}
	m(E^{m,i,+}_{\lambda',\delta}) = \hat y_i^+(\lambda') - \tilde y_i 
	\;\overset{(\ref{zzz})}{<}\; \Bigl(\tfrac{\tilde k_i}{\kappa_i}\Bigr)^{1/\tilde k_i}\delta.
\end{equation}

\quad Similarly,
\begin{equation}\label{fff}
	\tilde y_i - \hat y_i^-(\lambda') < \Bigl(\tfrac{\tilde k_i}{\kappa_i}\Bigr)^{1/\tilde k_i}\delta,
\end{equation}
and hence
\begin{equation}\label{ffff}
	m(E^{m,i,-}_{\lambda',\delta}) 
	< \Bigl(\tfrac{\tilde k_i}{\kappa_i}\Bigr)^{1/\tilde k_i}\delta.
\end{equation}

Combining \eqref{jfa}, we conclude that for 
$\lambda' \in (\lambda - \delta^{\tilde k_i}, \lambda + \delta^{\tilde k_i})$,
\begin{equation}\label{qhh}
	m(E^{m,i}_{\lambda',\delta}) 
	= m(E^{m,i,-}_{\lambda',\delta}) + m(E^{m,i,+}_{\lambda',\delta})
	< 2\Bigl(\tfrac{\tilde k_i}{\kappa_i}\Bigr)^{1/\tilde k_i}\delta.
\end{equation}

\item[\textbf{(B.3)}]\textbf{Case~c.3.} Finally, for Case~c.3, when $\lambda' \leq \lambda - \delta^{\tilde k_i}$, 
the monotonicity of $v$ implies
\begin{equation}\label{kj}
	E^{m,i,-}_{\lambda',\delta} = E^{m,i,+}_{\lambda',\delta} = \varnothing,
\end{equation}
and therefore
\begin{equation}\label{ling}
	m(E^{m,i,-}_{\lambda',\delta}) = m(E^{m,i,+}_{\lambda',\delta}) = 0.
\end{equation}
\end{itemize}

As in the case where $\tilde k_i$ is odd, define
\[
C_i(\lambda) := \max\left\{\, 
  4\Bigl(\tfrac{\tilde k_i}{\kappa_i}\Bigr)\Bigl(\tfrac{\kappa_i'}{\tilde k_i}\Bigr)^{1 - 1/\tilde k_i},\;
  2\Bigl(\tfrac{2\tilde k_i}{\kappa_i}\Bigr)^{1/\tilde k_i}
  \right\}.
\]
Combining \eqref{jfa}, \eqref{xg}, \eqref{jyb}, \eqref{xgg}, \eqref{jybb}, 
\eqref{yyc}, \eqref{zjd}, \eqref{xj}, \eqref{ffff}, \eqref{qhh}, 
\eqref{kj}, and \eqref{ling}, we conclude that for any 
$\delta \in (0, \hat{\delta}_i(\lambda)]$, there exist intervals 
$J_i(\lambda,\delta)$ and $K_i(\lambda,\delta)$ 
(here the empty set is also regarded as a special case of an interval) such that
\begin{align}\label{kl}
	\begin{split}
	&E^{m,i}_{\lambda',\delta} \subset J_i(\lambda,\delta) \cup K_i(\lambda,\delta),\\
	&m(J_i(\lambda,\delta)) \leq \tfrac{C_i(\lambda)}{2}\,\delta,\quad m(K_i(\lambda,\delta)) \leq \tfrac{C_i(\lambda)}{2}\,\delta,\\
	&m(E^{m,i}_{\lambda',\delta}) \leq m(J_i(\lambda,\delta)) + m(K_i(\lambda,\delta))\leq C_i(\lambda)\delta.
	\end{split}
\end{align}

\textit{The above proof was given under the assumptions 
$\tilde y_1 > L_1$ and $\tilde y_{N(\lambda)} < L_2$, 
which correspond to Case~a.4 in \eqref{si}.  
For the other three cases, the same argument shows that \eqref{kl} still holds. Moreover, the proof was carried out under the assumption 
$v^{(\tilde k_i)}(\tilde y_i) > 0$.  
However, when $v^{(\tilde k_i)}(\tilde y_i) < 0$, 
\eqref{kl} remains valid as well.}
\end{itemize}

Combining the result for odd $\tilde k_i$ in \eqref{cf} with the result for even $\tilde k_i$ in \eqref{kl}, we obtain 
$\forall i\in\{1,\ldots,N(\lambda)\},\ \exists\, C_i(\lambda) > 0, \ \hat{\delta}_i(\lambda) > 0, \ 
\forall\, \lambda' \in \bigl(\lambda - (\hat{\delta}_i(\lambda))^m, \, \lambda + (\hat{\delta}_i(\lambda))^m \bigr), 
\forall\, \delta \in (0,\hat{\delta}_i(\lambda)],$
\begin{equation}\label{jl}
	m(E^{m,i}_{\lambda',\delta}) \leq C_i(\lambda)\,\delta.
\end{equation}

Define 
\begin{align}
\label{d}
	&\delta_0(\lambda) :=\min\left\{\, \hat{\delta}_1(\lambda), \ldots, \hat{\delta}_{N(\lambda)}(\lambda) \,\right\},\\
  \label{qiuhe}  &C(\lambda) :=\sum_{i=1}^{N(\lambda)} C_i(\lambda).
\end{align}
 Then
\[
	m(E^m_{\lambda',\delta}) 
	\overset{(\ref{qh})}{=} \sum_{i=1}^{N(\lambda)} m(E^{m,i}_{\lambda',\delta})
	\overset{\eqref{kl}}{\leq} \sum_{i=1}^{N(\lambda)} C_i(\lambda)\,\delta
	\overset{\eqref{d}}{=} C(\lambda)\,\delta.
\]
The proof of the lemma is thus complete.
\end{proof}

Based on the result of Lemma \ref{prior estimate for E(lambda)}, we are able to obtain the uniform boundedness of $m(E^m_{\lambda,\delta})$ with respect to $\lambda$.

\begin{lemma}\label{|E(lambda)|}
Let $d=1$, $m \geq 1$, and $L_1, L_2$ are arbitrary real numbers with $L_1 < L_2$. Suppose $v:[L_1,L_2] \to \mathbb{R}$ is a $C^m$ function satisfying the $m$-th order non-degeneracy condition
\[
 |v'(y)| + |v''(y)| + \cdots + |v^{(m)}(y)| > 0,\quad \forall\, y \in [L_1,L_2].
\]
Then there exist constants $C_0 > 0$ and $\hat{\delta}_0 \in (0,1)$ such that, 
for any $\delta \in (0,\hat{\delta}_0]$ and any $\lambda \in \mathbb{R}$, we have
\[
 m(E^m_{\lambda,\delta})\leq C_0 \delta .
\]
\end{lemma}
\begin{proof}Let us recall the definitions of $v_1$ and $v_2$ in Lemma~\ref{prior estimate for E(lambda)}, namely, $v_1 = \min_{y \in [L_1,L_2]} v(y), \ 
v_2 = \max_{y \in [L_1,L_2]} v(y)$. Noticing that
\[
  [v_1,v_2] \subset \bigcup_{\lambda \in [v_1,v_2]} 
  \bigl(\lambda - (\delta_0(\lambda))^m,\;\lambda + (\delta_0(\lambda))^m \bigr),
\]
by the finite covering theorem there exist $M \in \mathbb{N}^*$ and 
$\lambda_1,\ldots,\lambda_M \in [v_1,v_2]$ such that
\begin{equation}\label{jh}
  [v_1,v_2] \subset \bigcup_{j=1}^M
  \bigl(\lambda_j - (\delta_0(\lambda_j))^m,\;\lambda_j + (\delta_0(\lambda_j))^m \bigr) =: B.
\end{equation}

Let
\begin{align}\label{C}
\begin{split}
  \alpha = \inf B, \quad\beta  = \sup B, \quad
  C_0    = \max\{C(\lambda_1),\ldots,C(\lambda_M)\},
\end{split}
\end{align}
then $C_0 > 0$.  Since $[v_1,v_2] \subset B$, we have $\alpha \leq v_1$.  

If $\alpha = v_1$, then $\inf B = v_1$. For any 
$j \in \{1,\ldots,M\}$,
\[
\lambda_j - (\delta_0(\lambda_j))^m
= \inf\bigl(\lambda_j - (\delta_0(\lambda_j))^m,\;\lambda_j + (\delta_0(\lambda_j))^m\bigr)
\geq \inf B = v_1,
\]
which implies that 
\[
v_1 \notin 
(\lambda_j - (\delta_0(\lambda_j))^m,\;\lambda_j + (\delta_0(\lambda_j))^m).
\]
Since $j$ is arbitrary, we deduce that
\[
v_1 \notin \bigcup_{j=1}^M
(\lambda_j - (\delta_0(\lambda_j))^m,\;\lambda_j + (\delta_0(\lambda_j))^m) = B,
\]
which contradicts $v_1 \in [v_1,v_2] \subset B$.  Therefore, $\alpha < v_1$, and similarly, $\beta > v_2$.

In fact, going further, one can prove that $B = (\alpha,\beta)$, and the proof
requires some knowledge of topology. Indeed, for any $y',y'' \in B$, there exist $i,j \in \{1,\ldots,M\}$ such that
\[
y' \in \bigl(\lambda_i - (\delta_0(\lambda_i))^m,\,\lambda_i + (\delta_0(\lambda_i))^m\bigr), 
\quad 
y'' \in \bigl(\lambda_j - (\delta_0(\lambda_j))^m,\,\lambda_j + (\delta_0(\lambda_j))^m\bigr).
\]
Define
\begin{align*}
\phi(t) &= 
\begin{cases}
 y' + 3t(\lambda_i - y'), & t \in [0,1/3], \\[0.5ex]
 \lambda_i + 3(t-1/3)(\lambda_j - \lambda_i), & t \in (1/3,2/3], \\[0.5ex]
 \lambda_j + 3(t-2/3)(y'' - \lambda_j), & t \in (2/3,1].
\end{cases}
\end{align*}
\begin{itemize}
    \item For $t \in [0,1/3]$, we have $\phi(t) \in (\lambda_i - (\delta_0(\lambda_i))^m,\,
\lambda_i + (\delta_0(\lambda_i))^m) \subset B$;
\item for $t \in (1/3,2/3]$, we have 
$\phi(t) \in [\min\{\lambda_i,\lambda_j\},\max\{\lambda_i,\lambda_j\}] 
\subset [v_1,v_2] \subset B$; 
\item for $t \in (2/3,1]$, we have 
$\phi(t) \in (\lambda_j - (\delta_0(\lambda_j))^m,\,\lambda_j + (\delta_0(\lambda_j))^m)
\subset B$.  
\end{itemize}
Thus $\phi$ is a path in $B$ connecting $y'$ and $y''$, hence $B$ is path-connected. What's more, as a union of some open intervals, $B$ is an open subset of $\mathbb{R}$.
Since the open path-connected subsets of $\mathbb{R}$ are open intervals, we conclude that
\begin{equation}\label{ab}
B = (\inf B, \sup B) = (\alpha,\beta).
\end{equation}

Now set 
\[
\hat{\delta}_0 = \tfrac{1}{2}\min\Bigl\{(v_1-\alpha)^{1/m},\,
(\beta-v_2)^{1/m},\,\delta_0(\lambda_1),\ldots,\delta_0(\lambda_M)\Bigr\},
\]
so that $\hat{\delta}_0 > 0$. For any $\lambda \in \mathbb{R}$ and any 
$\delta \in (0,\hat{\delta}_0]$, we have
\begin{equation}\label{zz}
m(E^m_{\lambda,\delta}) \leq C_0 \delta.
\end{equation}

\begin{itemize}
    \item[(1)] Indeed, if $\lambda \in (\alpha,\beta)$, then there exists $j' \in \{1,\ldots,M\}$ such that
$\lambda \in (\lambda_{j'} - (\delta_0(\lambda_{j'}))^m,\,
\lambda_{j'} + (\delta_0(\lambda_{j'}))^m)$, which implies $m(E^m_{\lambda,\delta})  \leq C_0 \delta$.
\item[(2)] If $\lambda \in \mathbb{R}\setminus(\alpha,\beta)$, then either $\lambda \leq \alpha$ 
or $\lambda \geq \beta$.
\begin{itemize}
    \item[(2.1)] If $\lambda \leq \alpha$, then for any $y \in [L_1,L_2]$, we have $
v(y)-\lambda \geq v_1 - \alpha \geq (2\hat{\delta}_0)^m > \delta^m$, hence $y \notin E^m_{\lambda,\delta}$, which gives 
\begin{align}
    \label{my}E^m_{\lambda,\delta}= \varnothing,\quad  m(E^m_{\lambda,\delta}) =0.
\end{align}
\item[(2.2)]If $\lambda \geq \beta$, then for any $y \in [L_1,L_2]$, we have $v(y)-\lambda \leq v_2 - \lambda \leq -(2\hat{\delta}_0)^m < -\delta^m$, hence $y \notin E^m_{\lambda,\delta}$, which gives 
\begin{align}
    \label{l}E^m_{\lambda,\delta}= \varnothing,\quad  m(E^m_{\lambda,\delta}) =0.\end{align}
\end{itemize}
\end{itemize}
Thus, \eqref{zz} holds, which completes the proof of the lemma.    
\end{proof}

Next, we obtain the following uniform bound for $m(\mathcal{E}^m_{\lambda,\delta})$ with respect to $\lambda$.

\begin{lemma}\label{mathcal{E}(lambda)}
Let $d=1$, $m \geq 1$, and $L_1, L_2$ are arbitrary real numbers with $L_1 < L_2$. Suppose $v:[L_1,L_2] \to \mathbb{R}$ is a $C^m$ function satisfying the $m$-th order non-degeneracy condition
\[
 |v'(y)| + |v''(y)| + \cdots + |v^{(m)}(y)| > 0,\quad \forall\, y \in [L_1,L_2].
\]
Then there exist constants $C_0' > 0$ and $\hat{\delta}_0 \in (0,1)$ such that, 
for any $\delta \in (0,\hat{\delta}_0]$ and any $\lambda \in \mathbb{R}$, we have
\[
 m(\mathcal{E}^m_{\lambda,\delta})\leq C_0' \delta .
\]
\end{lemma}
\begin{proof}Define $C_0':\,=\; N_0\,(C_0+4),$ where $N_0$ is given by Lemma \ref{Uniformly finitely many intersection points}. From \eqref{jh} and \eqref{ab}, we obtain
\begin{equation}\label{bjj}
  [v_1,v_2]\;\subset\;\bigcup_{j=1}^M
  \bigl(\lambda_j-(\delta_0(\lambda_j))^m,\;\lambda_j+(\delta_0(\lambda_j))^m\bigr)
  \;=\;(\alpha,\beta).
\end{equation}
Fix any $\lambda\in\R$ and any $\delta\in(0,\hat\delta_0],$ where $\hat{\delta_0}$ is given in Lemma \ref{|E(lambda)|}.

If $\lambda\in(-\infty,\alpha]\cup[\beta,+\infty)$, then by \eqref{my} and \eqref{l} we have
$E^m_{\lambda,\delta}= \varnothing$, hence
\begin{align*}
  \mathcal{E}^m_{\lambda,\delta}
  = \bigl\{y\in[L_1,L_2]: \mathrm{dist}(y,E^m_{\lambda,\delta})<\delta \bigr\}
  = \bigl\{y\in[L_1,L_2]: \mathrm{dist}(y,\varnothing)<\delta \bigr\}
  = \varnothing,
\end{align*}
so that $m(\mathcal{E}^m_{\lambda,\delta})=0$.

If $\lambda\in(\alpha,\beta)$, then by \eqref{jh} and \eqref{ab} there exists
$j''\in\{1,\ldots,M\}$ such that
\begin{equation}\label{bh}
  \lambda\in\bigl(\lambda_{j''}-(\delta_0(\lambda_{j''}))^m,\;
                    \lambda_{j''}+(\delta_0(\lambda_{j''}))^m\bigr).
\end{equation}
By \eqref{el} and \eqref{jj}, there exist $\varepsilon>0$ and $N(\lambda)\in \mathbb{N}^*$ as well as points
$\tilde{y}_1,\tilde{y}_2,\ldots,\tilde{y}_{N(\lambda)}\in[L_1,L_2]$ with
\[
  \tilde{y}_1<\tilde{y}_2<\cdots<\tilde{y}_{N(\lambda)}
\]
such that 
$$E^m_{\lambda,\delta}\;=\;\bigcup_{i=1}^{N(\lambda)} E^{m,i}_{\lambda,\delta}=\;\bigcup_{i=1}^{N(\lambda)} \left((\tilde{y}_i-\eps,\tilde{y}_i+\eps)\cap E^{m}_{\lambda,\delta}\right).$$ Hence,
\begin{align}\label{cc}
\begin{split}
  \mathcal{E}^m_{\lambda,\delta}
  &= \Bigl\{y\in[L_1,L_2]:\mathrm{dist}(y,E^m_{\lambda,\delta})<\delta\Bigr\}= \Bigl\{y\in[L_1,L_2]:
\mathrm{dist}\!\Bigl(y,\bigcup_{i=1}^{N(\lambda)}E^{m,i}_{\lambda,\delta}\Bigr)<\delta\Bigr\}\\
  &= \bigcup_{i=1}^{N(\lambda)}
     \Bigl\{y\in[L_1,L_2]:\mathrm{dist}\bigl(y,E^{m,i}_{\lambda,\delta}\bigr)<\delta\Bigr\}\\
  &\subset \bigcup_{i=1}^{N(\lambda)}
     \Bigl\{y\in\R:\mathrm{dist}\bigl(y,E^{m,i}_{\lambda,\delta}\bigr)<\delta\Bigr\}
     \;=: \bigcup_{i=1}^{N(\lambda)} H_i(\lambda,\delta).
\end{split}
\end{align}

According to \eqref{cf} and \eqref{jh}, for each $i\in\{1,\ldots,N(\lambda)\}$ there are two
cases:
\begin{align}\label{er}
\begin{split}
 &\text{Case d.1: } \tilde{k}_i \text{ is odd;}\\
 &\text{Case d.2: } \tilde{k}_i \text{ is even.}
\end{split}
\end{align}
Here $\tilde{k}_i$ is defined in \eqref{ki}. 

In Case d.1 (i.e. $\tilde{k}_i$ odd), by \eqref{bh},
\[
  \lambda\in\bigl(\lambda_{j''}-(\delta_0(\lambda_{j''}))^m,\;
                   \lambda_{j''}+(\delta_0(\lambda_{j''}))^m\bigr)
  \subset \bigl(\lambda_{j''}-\hat\delta_i(\lambda_{j''}),\;
                \lambda_{j''}+\hat\delta_i(\lambda_{j''})\bigr),
\]
hence by \eqref{cf} there exists an interval $I_i(\lambda_{j''},\delta)$ such that
\begin{equation}\label{ji}
  E^{m,i}_{\lambda,\delta}\subset I_i(\lambda_{j''},\delta),
  \qquad |I_i(\lambda_{j''},\delta)|
  \le C_i(\lambda_{j''})\,\delta
  \overset{\eqref{qiuhe}}{\le} C(\lambda_{j''})\,\delta
  \overset{\eqref{C}}{\le} C_0\,\delta.
\end{equation}
Therefore, by \eqref{cc},
\begin{equation}\label{h}
  H_i(\lambda,\delta)\subset
  \{y\in\R:\mathrm{dist}(y,I_i(\lambda_{j''},\delta))<\delta\}
  =:\widehat I_i(\lambda_{j''},\delta),
\end{equation}
so $\widehat I_i(\lambda_{j''},\delta)$ is also an interval and
\begin{equation}\label{hi}
  |H_i(\lambda,\delta)|
  \overset{\eqref{h}}{\le}|\widehat I_i(\lambda_{j''},\delta)|
  \le |I_i(\lambda_{j''},\delta)|+2\delta
  \overset{\eqref{ji}}{\le}(C_0+2)\delta
  <(C_0+4)\delta.
\end{equation}

In Case d.2 (i.e. $\tilde{k}_i$ even), again by \eqref{bh},
\[
  \lambda\in\bigl(\lambda_{j''}-(\delta_0(\lambda_{j''}))^m,\;
                   \lambda_{j''}+(\delta_0(\lambda_{j''}))^m\bigr)
  \subset \bigl(\lambda_{j''}-\hat\delta_i(\lambda_{j''}),\;
                \lambda_{j''}+\hat\delta_i(\lambda_{j''})\bigr),
\]
and by \eqref{kl} there exist sets $K_i(\lambda_{j''},\delta)$ and $J_i(\lambda_{j''},\delta)$
such that
\begin{equation}\label{bjjj}
  E^{m,i}_{\lambda,\delta}\subset
  K_i(\lambda_{j''},\delta)\cup J_i(\lambda_{j''},\delta),
\end{equation}
and
\begin{align}\label{kjjj}
\begin{split}
  |K_i(\lambda_{j''},\delta)|
    &\le \frac{C_i(\lambda_{j''})}{2}\,\delta
     \overset{\eqref{qiuhe}}{\le} \frac{C(\lambda_{j''})}{2}\,\delta
     \le \frac{C_0}{2}\,\delta,\\
  |J_i(\lambda_{j''},\delta)|
    &\le \frac{C_i(\lambda_{j''})}{2}\,\delta
     \overset{\eqref{qiuhe}}{\le} \frac{C(\lambda_{j''})}{2}\,\delta
     \le \frac{C_0}{2}\,\delta.
\end{split}
\end{align}
Using \eqref{cc} and \eqref{bjjj}, we have
\begin{align}\label{H}
\begin{split}
  H_i(\lambda,\delta)
  &= \{y\in\R:\mathrm{dist}(y,E^{m,i}_{\lambda,\delta})<\delta\}\\
  &\subset \{y\in\R:\mathrm{dist}(y,K_i(\lambda_{j''},\delta)\cup J_i(\lambda_{j''},\delta))<\delta\}\\
  &= \{y\in\R:\mathrm{dist}(y,K_i(\lambda_{j''},\delta))<\delta\}
     \;\cup\;
     \{y\in\R:\mathrm{dist}(y,J_i(\lambda_{j''},\delta))<\delta\}\\
  &=: \widehat K_i(\lambda_{j''},\delta)\,\cup\,\widehat J_i(\lambda_{j''},\delta),
\end{split}
\end{align}
where $\widehat K_i(\lambda_{j''},\delta)$ and $\widehat J_i(\lambda_{j''},\delta)$ are
intervals satisfying
\begin{align}\label{jk}
\begin{split}
  |\widehat K_i(\lambda_{j''},\delta)|
    &\le |K_i(\lambda_{j''},\delta)|+2\delta
     \overset{\eqref{kjjj}}{\le} \Bigl(\frac{C_0}{2}+2\Bigr)\delta,\\
  |\widehat J_i(\lambda_{j''},\delta)|
    &\le |J_i(\lambda_{j''},\delta)|+2\delta
     \overset{\eqref{kjjj}}{\le} \Bigl(\frac{C_0}{2}+2\Bigr)\delta.
\end{split}
\end{align}
Therefore,
\begin{equation}\label{hkj}
  |H_i(\lambda,\delta)|
  \overset{\eqref{H}}{\le}
  |\widehat K_i(\lambda_{j''},\delta)|+|\widehat J_i(\lambda_{j''},\delta)|
  \overset{\eqref{jk}}{\le} (C_0+4)\,\delta.
\end{equation}

Consequently,
\begin{align*}
  m(\mathcal{E}^m_{\lambda,\delta})
   &\overset{\eqref{cc}}{\le} \sum_{i=1}^{N(\lambda)} |H_i(\lambda,\delta)|
    \;\overset{\eqref{hi},\eqref{hkj}}{\le}\;
      N(\lambda)\,(C_0+4)\,\delta
    \;\overset{\textrm{Lemma}\ \ref{Uniformly finitely many intersection points}}{\le}\;
      N_0\,(C_0+4)\,\delta
    \;=\; C_0'\,\delta.
\end{align*}
This completes the proof of Lemma \ref{mathcal{E}(lambda)}.
\end{proof}

\section{Uniform-in-$\lambda$ estimate of $m(\mathcal{E}_{\lambda,\delta}^m)$ in non-compact regions}
\label{sec3}

With the preparations from Lemmas \ref{piecewise strict monotonicity}–\ref{mathcal{E}(lambda)} in place, we now return to the proof of Proposition \ref{mathcal e measure-prop}.
\begin{proposition}\label{mathcal e measure-prop}
Assume $d=1$ and $v : \Omega\to \mathbb{R}$ is a $C^{m}$ function, where $m$ and $\Omega$ are the same as in Theorem \ref{thm:main1}, and $v$ satisfies conditions \eqref{eq:vder} and \eqref{non-degenerate at infinity} of Theorem \ref{thm:main1}. Then
there exist constant $C > 0$ and  $\delta_0>0$ such that, for all $\nu > 0$, all $k \neq 0$ and all $\lambda\in\mathbb{R}\ and\ \delta\in(0,\delta_0],$ 
\begin{equation}\label{mathcal e measure}
m(\mathcal{E}_{\lambda,\delta}^m) \leq C\delta.
\end{equation}
\end{proposition}
\begin{proof}[\bf Proof of Proposition \ref{mathcal e measure-prop}]
Without loss of generality, we present the proof for $\Omega=\mathbb{R}$.

Since $\varliminf_{|y|\to\infty}|v'(y)|\ge c_0>0$, there exists $\Delta>0$ with
\begin{equation}\label{xjx}
  |v'(y)| \ge \tfrac{c_0}{2} \quad \text{for all } |y|\ge \Delta.
\end{equation}

We now prove that $m(\mathcal{E}^{m}_{\lambda,\delta}) = O(\delta)$ uniformly for $\lambda\in\mathbb{R}$; that is, there exist constants $C_0>0$ and $\delta_0>0$ such that, for any $\lambda\in\mathbb{R}$ and any $\delta\in(0,\delta_0]$,
\[
  m(\mathcal{E}^{m}_{\lambda,\delta}) \le C_0\,\delta.
\]

From \eqref{xjx}, either $ v'(y)\ge \tfrac{c_0}{2} \text{ or } v'(y)\le -\tfrac{c_0}{2} (\forall y\in[\Delta,\infty))$. Without loss of generality, assume that $v'(y)\ge \tfrac{c_0}{2}$ holds. Then $v$ is strictly increasing on $[\Delta,\infty)$, and for all $y\ge \Delta$,
\[
  v(y)=v(\Delta)+\int_{\Delta}^{y} v'(z)\,dz
  \;\ge\; v(\Delta)+\frac{c_0}{2}\,(y-\Delta),
\]
hence $v(+\infty)=+\infty$.

For any $\delta\in(0,\Delta]$ and any $\lambda\in\mathbb{R}$, set
\begin{align}\label{FF}
\begin{split}
  E^{m}_{\lambda,\delta}
  = &\bigl(E^{m}_{\lambda,\delta}\cap(-\infty,-2\Delta]\bigr)
     \;\cup\; \bigl(E^{m}_{\lambda,\delta}\cap[-2\Delta,2\Delta]\bigr)
     \;\cup\; \bigl(E^{m}_{\lambda,\delta}\cap[2\Delta,\infty)\bigr) \\
   =&:F_1 \cup F_2 \cup F_3 .
\end{split}
\end{align}
Define
\begin{align}\label{G}
\begin{split}
 & \mathcal{E}^{m}_{\lambda,\delta}
  = \{y\in\mathbb{R}:\mathrm{dist}(y,E^{m}_{\lambda,\delta})<\delta\} \\
  = &\{y\in\mathbb{R}:\mathrm{dist}(y,F_1)<\delta\}
     \;\cup\; \{y\in\mathbb{R}:\mathrm{dist}(y,F_2)<\delta\}
     \;\cup\; \{y\in\mathbb{R}:\mathrm{dist}(y,F_3)<\delta\} \\
 =&: G_1 \cup G_2 \cup G_3 .
\end{split}
\end{align}
Moreover, define
\begin{align}\label{hh}
\begin{split}
  H_1 &:= E^{m}_{\lambda,\delta}\cap(-\infty,-\Delta)
         \;\supset\; E^{m}_{\lambda,\delta}\cap(-\infty,-2\Delta] \;=\; F_1, \\
  H_2 &:= E^{m}_{\lambda,\delta}\cap(-3\Delta,3\Delta)
         \;\supset\; E^{m}_{\lambda,\delta}\cap[-2\Delta,2\Delta] \;=\; F_2, \\
  H_3 &:= E^{m}_{\lambda,\delta}\cap(\Delta,\infty)
         \;\supset\; E^{m}_{\lambda,\delta}\cap[2\Delta,\infty) \;=\; F_3 .
\end{split}
\end{align}
Then
\begin{align}\label{GG}
\begin{split}
  &G_1 \subset (-\infty,-2\Delta+\delta) \subset (-\infty,-\Delta),\\
  &G_2 \subset (-2\Delta-\delta,\,2\Delta+\delta) \subset (-3\Delta,3\Delta),\\
  &G_3 \subset (2\Delta-\delta,\infty) \subset (\Delta,\infty).
\end{split}
\end{align}
Consequently,
\begin{align}\label{gs}
\begin{split}
  G_3
  &= \{y\in\mathbb{R}:\mathrm{dist}(y,F_3)<\delta\} \overset{\eqref{GG}}{=} \{y\in(\Delta,\infty):\mathrm{dist}(y,F_3)<\delta\} \\
  &\overset{\eqref{hh}}{\subset} \{y\in(\Delta,\infty):\mathrm{dist}(y,H_3)<\delta\}.
\end{split}
\end{align}

We distinguish the following three cases:
\begin{align}\label{sz}
\begin{split}
  &\text{Case e.1: } v(\Delta)\le \lambda-\delta,\\
  &\text{Case e.2: } \lambda-\delta< v(\Delta)< \lambda+\delta,\\
  &\text{Case e.3: } \lambda+\delta\le v(\Delta).
\end{split}
\end{align}

\begin{itemize}
    \item For Case e.1, by the strict monotonicity of $v$,
there exists a unique pair $(\Delta_1,\Delta_2)$ with $\Delta_1\ge \Delta$ and
$\Delta_2\ge \Delta$ such that
\begin{equation}\label{dd}
  v(\Delta_1)=\lambda-\delta,\qquad v(\Delta_2)=\lambda+\delta,
\end{equation}
and moreover
\begin{equation}\label{jz}
  \Delta_1<\Delta_2,\qquad H_3=(\Delta_1,\Delta_2).
\end{equation}
Hence, by \eqref{gs} and \eqref{jz},
\[
  G_3 \subset \{\,y\in(\Delta,\infty): \text{dist}\bigl(y,(\Delta_1,\Delta_2)\bigr)<\delta \,\}
      \subset (\Delta_1-\delta,\Delta_2+\delta),
\]
and thus
\begin{equation}\label{dede}
  m(G_3) \le (\Delta_2+\delta)-(\Delta_1-\delta) = 2\delta + \Delta_2-\Delta_1.
\end{equation}
Moreover,
\begin{align*}
  2\delta
  &= (\lambda+\delta)-(\lambda-\delta)
   \overset{\eqref{dd}}{=} v(\Delta_2)-v(\Delta_1)
   = \int_{\Delta_1}^{\Delta_2} v'(y)\,dy
   \ge \int_{\Delta_1}^{\Delta_2} \frac{c_0}{2}\,dy
   = \frac{c_0}{2}\,(\Delta_2-\Delta_1),
\end{align*}
whence
\begin{equation}\label{ddd}
  \Delta_2-\Delta_1 \le \frac{4}{c_0}\,\delta.
\end{equation}
Combining \eqref{dede} and \eqref{ddd}, we obtain
\begin{equation}\label{ggg}
  m(G_3) \le \Bigl(2+\frac{4}{c_0}\Bigr)\delta.
\end{equation}
\item For Case e.2, by the strict monotonicity of $v$,
there exists a unique $\Delta_3>\Delta$ such that
\[
  v(\Delta_3)=\lambda+\delta,
  \qquad H_3=(\Delta,\Delta_3).
\]
Hence, by \eqref{gs} and \eqref{jz},
\[
  G_3 \subset \{\,y\in(\Delta,\infty): \text{dist}\bigl(y,(\Delta,\Delta_3)\bigr)<\delta \,\}
      \subset (\Delta-\delta,\Delta_3+\delta),
\]
and thus
\begin{equation}\label{dedede}
  m(G_3) \le (\Delta_3+\delta)-(\Delta-\delta) = 2\delta + \Delta_3-\Delta.
\end{equation}
Moreover, using $v(\Delta_3)=\lambda+\delta$ and $\lambda-\delta< v(\Delta)$,
\[
  v(\Delta_3)-v(\Delta) \le (\lambda+\delta)-(\lambda-\delta)=2\delta,
\]
then 
\[
  v(\Delta_3)-v(\Delta)=\int_{\Delta}^{\Delta_3} v'(y)\,dy
  \ge \int_{\Delta}^{\Delta_3} \frac{c_0}{2}\,dy
  = \frac{c_0}{2}\,(\Delta_3-\Delta).
\]
Hence
\begin{equation}\label{dxx}
  \Delta_3-\Delta \le \frac{4}{c_0}\,\delta.
\end{equation}
Combining \eqref{dedede} and \eqref{dxx}, we obtain
\begin{equation}\label{gsgs}
  m(G_3) \le \Bigl(2+\frac{4}{c_0}\Bigr)\delta.
\end{equation}
\item For Case e.3, by the monotonicity of $v$, since for all $y\ge 2\Delta$ we have
\( v(y) > v(\Delta) \ge \lambda+\delta \),
it follows that \(F_3=\varnothing\). Hence, by \eqref{G},
\begin{equation}\label{lin}
  G_3=\varnothing,\qquad m(G_3)=0.
\end{equation}

\end{itemize}

Therefore, combining \eqref{ggg}, \eqref{gsgs}, and \eqref{lin}, for any $ \delta\in(0,\Delta]$ and $\lambda\in\mathbb{R}$, regardless of which case in \eqref{sz} occurs, we have
\begin{equation}\label{gggg}
  m(G_3)\le \Bigl(2+\frac{4}{c_0}\Bigr)\delta.
\end{equation}

Similarly, for any $\delta\in(0,\Delta]$ and any $\lambda\in\mathbb{R}$,
\begin{equation}\label{ggggg}
  m(G_1)\le \Bigl(2+\frac{4}{c_0}\Bigr)\delta.
\end{equation}

For any $y\in[-3\Delta,3\Delta]\subset\mathbb{R}$, the $m$-th order non-degeneracy condition holds:
\[
  |v'(y)|+\cdots+|v^{(m)}(y)|>0.
\]
By Lemma \ref{mathcal{E}(lambda)}, there exist constants $C_0'>0$ and $\hat{\delta}_0>0$ such that,
for any $\delta\in(0,\hat{\delta}_0]$ and any $\lambda\in\mathbb{R}$,
\begin{equation}\label{gg}
  m(G_2)\le C_0'\,\delta .
\end{equation}

Therefore, combining \eqref{lin}, \eqref{ggggg}, and \eqref{gg}, by taking $$\delta_0=\min\{\Delta,\hat{\delta}_0\},\quad C=4+\f{8}{c_0}+C_0',$$ we obtain that for any $\delta\in(0,\delta_0]$ and any $\lambda\in\mathbb{R},$
\begin{align*}
  m(\mathcal{E}^{m}_{\lambda,\delta})
 \le m(G_1)+m(G_2)+m(G_3)
  \le \Bigl(4+\frac{8}{c_0}+C_0'\Bigr)\delta=C\de.
\end{align*}

This completes the proof of Proposition \ref{mathcal e measure-prop}.
\end{proof}

\section{Proof of the higher-dimensional case}
This section is devoted to the proof of Theorem~\ref{thm:higher-d}.
\begin{proof}[\bf Proof of Theorem~\ref{thm:higher-d}]
For each $j\in\{1,\ldots,d\}$, set
\[
  H_{\nu,k,j}\;=\; -\,\nu\,\partial_{y_j}^2 \;+\; \mathrm{i}\,k\,v_j(y_j),
  \qquad\text{so that}\qquad
  H_{\nu,k}\;=\;\sum_{j=1}^d H_{\nu,k,j}.
\]
Hence, for any $t\ge0$,
\begin{equation}\label{gkt}
  g(k,t)\;=\;\mathrm{e}^{-tH_{\nu,k}}g_0 \;=\; \mathrm{e}^{-t\sum_{j=1}^d H_{\nu,k,j}}\,g_0.
\end{equation}
The identity \eqref{gkt} holds because the operators $H_{\nu,k,j}$ commute pairwise; indeed,
\[
  [H_{\nu,k,j},\,H_{\nu,k,l}] = 0 \qquad \text{for all } j, l\in\{1,\ldots,d\}.
\]

By assumptions \textup{(i)}–\textup{(ii)} (see Theorem~\ref{thm:higher-d}),
when $\Omega_j=(L_{3,j},L_{4,j})$ we apply Lemma~3 (compact case), and when
$\Omega_j\in\{(-\infty,\infty),\,(-\infty,L_{1,j}),\,(L_{2,j},\infty)\}$ we apply
Theorem~1 (non-compact case). In both situations there exist constants $C_{1,j},C_{2,j}>0$
such that the one–dimensional semigroup satisfies the decay bound with rate $\lambda_{\nu,k,j}$.

Iterating this bound coordinate by coordinate and using Fubini, we obtain
\begin{align*}
 \|g(k,t)\|_{L^2(\Omega)}^2
 &= \Bigl\| \mathrm{e}^{-t\sum_{j=1}^d H_{\nu,k,j}} g_0 \Bigr\|_{L^2(\Omega)}^2
   &&\text{by \eqref{gkt}}\\
 &= \int_{\prod_{j=2}^d \Omega_j}
    \Biggl[\int_{\Omega_1}
      \Bigl|\,\mathrm{e}^{-tH_{\nu,k,1}}\!\Bigl(\mathrm{e}^{-t\sum_{j=2}^d H_{\nu,k,j}} g_0\Bigr)\Bigr|^2 \,dy_1
    \Biggr] dy_2\cdots dy_d\\
 &\le C_{1,1}^2\,\mathrm{e}^{-2C_{2,1}\lambda_{\nu,k,1}t}
    \int_{\prod_{j=2}^d \Omega_j}
      \Biggl[\int_{\Omega_1}
        \Bigl|\,\mathrm{e}^{-t\sum_{j=2}^d H_{\nu,k,j}} g_0\Bigr|^2 \,dy_1
      \Biggr] dy_2\cdots dy_d\\
 &= C_{1,1}^2\,\mathrm{e}^{-2C_{2,1}\lambda_{\nu,k,1}t}\,
    \Bigl\| \mathrm{e}^{-t\sum_{j=2}^d H_{\nu,k,j}} g_0 \Bigr\|_{L^2(\Omega)}^2\\
 &\le (C_{1,1}C_{1,2})^2
     \mathrm{e}^{-2(C_{2,1}\lambda_{\nu,k,1}+C_{2,2}\lambda_{\nu,k,2})t}
     \Bigl\| \mathrm{e}^{-t\sum_{j=3}^d H_{\nu,k,j}} g_0 \Bigr\|_{L^2(\Omega)}^2\\
 &\le \cdots \\
 &\le \Bigl(\prod_{j=1}^d C_{1,j}\Bigr)^2
     \exp\Bigl(-2\bigl(C_{2,1}\lambda_{\nu,k,1}+\cdots+C_{2,d}\lambda_{\nu,k,d}\bigr)t\Bigr)
     \,\|g_0\|_{L^2(\Omega)}^2\\
 &\le C^2\,\exp\Bigl(-2c\Bigl(\textstyle\sum_{j=1}^d \lambda_{\nu,k,j}\Bigr)t\Bigr)\,
     \|g_0\|_{L^2(\Omega)}^2,
\end{align*}
where
\[
  C \;=\; \prod_{j=1}^d C_{1,j},
  \qquad
  c \;=\; \min\{C_{2,1},\ldots,C_{2,d}\}.
\]
Taking square roots yields
\[
  \|g(k,t)\|_{L^2(\Omega)}
  \;\le\;
  C\,\exp\Bigl(-\,c\Bigl(\textstyle\sum_{j=1}^d \lambda_{\nu,k,j}\Bigr)t\Bigr)\,
  \|g_0\|_{L^2(\Omega)},
\]
which is the desired estimate.
\end{proof}

\appendix
\section{Basic Calculus}
\begin{lemma}\label{gx}
For any $g \in H^1(\Omega)$ (where $\Omega = \mathbb{R},\,(-\infty,L_{1}],\,[L_{2},\infty)$), one has
\begin{align*}
\|g\|^{2}_{L^{\infty}}
\leq 2 \,\|g\|_{L^{2}}\,\|g'\|_{L^{2}}.
\end{align*}  
\end{lemma}
\begin{proof}
  Because for any $g \in H^1(\Omega)$ we have $g(-\infty)=0$ (or $g(+\infty)=0$), it follows that for any $x \in \Omega$,
\begin{align*}
\begin{split}
|g(x)|^2
&= \int_{-\infty}^x \bigl(g(y)\overline{g(y)}\bigr)'\,\mathrm{d}y = \int_{-\infty}^x g(y)\overline{g'(y)} + g'(y)\overline{g(y)}\,\mathrm{d}y \\
&\leq 2\int_{-\infty}^x |g(y)|\,|g'(y)|\,\mathrm{d}y \leq 2\|g\|_{L^2}\,\|g'\|_{L^2}.
\end{split}
\end{align*}
\end{proof}

\section*{Acknowledgement}
T. Li is partially supported by  National Natural Science Foundation of China under Grant 12421001.

\section*{Declarations}

\subsection*{Conflict of interest} The authors declare that there are no conflicts of interest.

\subsection*{Data availability}
This article has no associated data.

\end{document}